\documentclass[10pt]{article}

\usepackage[english]{babel}
\usepackage{amsmath}
\usepackage{amssymb}
\usepackage{bm}
\usepackage{bbm}
\usepackage{mathrsfs,color}
\usepackage{graphics,graphicx,theorem}
\usepackage{dsfont}
\usepackage{setspace}

\usepackage{mathtools}

\usepackage{natbib}
\usepackage{multirow}
\usepackage{hhline}
\usepackage{hyperref}

\usepackage{authblk}

\hypersetup{
	colorlinks=true,
	urlcolor=blue,
	citecolor=blue,
	linktoc=all,
	linkcolor=red}

\usepackage[scriptsize]{subfigure}
\allowdisplaybreaks

\makeatletter
\renewcommand\@biblabel[1]{}
\makeatother
\usepackage{amsmath}
\usepackage{bm}
\usepackage{amsfonts,dsfont,mathrsfs}

\usepackage{booktabs}
\usepackage{pgfplots}
\usetikzlibrary{backgrounds,calc,positioning,arrows}

\newcommand{\ub}{\mathbf{u}}

\newcommand{\Bcr}{\mathscr{B}}

\newcommand{\Wc}{\mathcal{W}}
\newcommand{\Pc}{\mathcal{P}}

\def\naturals{\mathbb{N}}
\def\reals{\mathbb{R}}

\newcommand{\probabilityspace}{(\Omega, \mathscr{F}, \mathds{P})}
\newcommand{\ppms}{\mathbb{P}}

\newcommand{\randommeasure}{\tilde{\mathfrak{p}}}

\newcommand{\empiric}{\mathfrak{e}_n}
\newcommand{\pfrak}{\mathfrak{p}}
\newcommand{\ee}{\mathsf{E}}
\newcommand{\pp}{\mathsf{P}}

\def\probabilityspace{(\Omega, \mathscr{F}, \mathsf{P})}
\def\ss{\mathbb{X}}
\def\ssa{\mathscr{X}}
\def\Space{\mathbb{S}}
\def\ps{\mathcal{P}(\ss)}
\def\psa{\mathscr{T}}
\newcommand{\ud}{\mathrm{d}}

\newcommand{\pms}{\mathcal{P}(\ss)}
\newcommand{\pmssa}{\mathscr{P}}
\newcommand{\Wp}{\mathcal{W}_p}
\newcommand{\Wuno}{\mathcal{W}_1}

\newcommand{\ind}{\mathds{1}}

\newtheorem{theorem}{Theorem}[section]
\newtheorem{definition}{Definition}[section]
\newtheorem{proposition}{Proposition}[section]
\newtheorem{corollary}{Corollary}[section]

\newtheorem{lemma}{Lemma}[section]

\newenvironment{proof}{\noindent \textit{Proof.}}{\hfill$\square$}

\usepackage[top=1.2in,bottom=1.2in,left=1.2in,right=1.2in]{geometry}

\begin{document}

\title{\bf {\Large{Wasserstein posterior contraction rates in non-dominated Bayesian nonparametric models}}}

\author[,1]{Federico Camerlenghi \thanks{Also affiliated to Collegio Carlo Alberto, Piazza V. Arbarello 8, Torino, and BIDSA, Bocconi University, Milano, Italy; federico.camerlenghi@unimib.it}}
\author[,2]{Emanuele Dolera \thanks{emanuele.dolera@unipv.it}}
\author[,3]{Stefano Favaro \thanks{Also affiliated to Collegio Carlo Alberto, Piazza V. Arbarello 8, Torino, and IMATI-CNR ``Enrico  Magenes", Milan, Italy; stefano.favaro@unito.it}}
\author[,4]{Edoardo Mainini \thanks{mainini@dime.unige.it}}

\affil[1]{Department of Economics, Management and Statistics, University of Milano-Bicocca, Milano, Italy.}
\affil[2]{Department of Mathematics, University of Pavia, Pavia, Italy.}
\affil[3]{Department of Economics and Statistics, University of Torino, Torino, Italy.}
\affil[4]{Department of Mechanical Engineering, University of Genova, Genova, Italy.}

\date{}
\maketitle
\thispagestyle{empty}

\setcounter{page}{1}

\begin{abstract}
Posterior contractions rates (PCRs) strengthen the notion of Bayesian consistency, quantifying the speed at which the posterior distribution concentrates on arbitrarily small neighborhoods of the true model, with probability tending to $1$ or almost surely, as the sample size goes to infinity. Under the Bayesian nonparametric framework, a common assumption in the study of PCRs is that the model is dominated for the observations; that is, it is assumed that the posterior can be written through the Bayes formula. In this paper, we consider the problem of establishing PCRs in Bayesian nonparametric models where the posterior distribution is not available through the Bayes formula, and hence models that are non-dominated for the observations. By means of the Wasserstein distance and a suitable sieve construction, our main result establishes PCRs in Bayesian nonparametric models where the posterior is available through a more general disintegration than the Bayes formula. To the best of our knowledge, this is the first general approach to provide PCRs in non-dominated Bayesian nonparametric models, and it relies on minimal modeling assumptions and on a suitable continuity assumption for the posterior distribution. Some refinements of our result are presented under additional assumptions on the prior distribution, and applications are given with respect to the Dirichlet process prior and the normalized extended Gamma process prior.
\end{abstract}


\maketitle


\section{Introduction}

Frequentist consistency of Bayesian procedures, or simply Bayesian consistency, guarantees that the posterior distribution concentrates on arbitrarily small neighborhoods of the true model, with probability tending to $1$ or almost surely, as the sample size goes to infinity \citep{Doo(49),Sch(65),Fre(63),Fre(65),Dia(86),Bar(99),Gho(99),W(04)}. See \cite[Chapter 6 and Chapter 7]{GV(00)} and references therein for a comprehensive and up-to-date account on Bayesian consistency. Posterior contractions rates (PCRs) strengthen the notion of Bayesian consistency, quantifying the speed at which such neighborhoods may decrease to zero meanwhile still capturing most of the posterior mass. The problem of establishing PCRs in finite-dimensional (parametric) Bayesian models have been first considered in \cite{IbHa(81)} and \cite{LeCam(86)}, providing optimal PCR under suitable modeling assumptions. However, it is in the seminal works \cite{GGV(00)} and \cite{ShWa(01)} that the problem of establishing PCRs have been considered in a systematic way, setting forth a general approach to provide PCRs in both finite-dimensional (parametric) and infinite-dimensional (nonparametric) Bayesian models. Since then, several approaches have been proposed and investigated in order to obtain more explicit and also sharper PCRs in Bayesian nonparametrics. Among them, we recall the metric entropy approach, in combination with the definition of specific tests \citep{Sch(65),GGV(00)}, approaches based on bracketing numbers and entropy integrals \citep{ShWa(01)}, the martingale approach \citep{W(04),WLP(07)}, the Hausdorff entropy approach \citep{Xing(10)} and approaches based on the Wasserstein distance \citep{deBlasi(20)} and on its ``dynamic" formulations in term of partial differential equations \citep{DFM(20)}. At the ground of most of these approaches there is the explicit construction of a sieve in the space of the parameters or, at least, the existence of a sieve is implied. We refer to \cite[Chapter 8 and Chapter 9]{GV(00)} and references therein for a comprehensive and up-to-date account on PCRs.

In this paper, we consider the problem of establishing PCRs in non-dominated Bayesian nonparametric models. Nonparametric priors, such as the Dirichlet process prior \citep{Fer(73)} and generalizations thereof \citep{Lij(10),GV(00)}, have been extensively investigated in the study of PCRs \citep[Section 8]{GGV(00)}. However, their use has been mainly as hierarchical priors, whereas the underlying model is assumed to be dominated for the observations; that is, it is assumed that the posterior distribution can be written through the Bayes formula. Here, we deal with Bayesian nonparametric models where the posterior distribution is not available through the Bayes formula, and hence models that are non-dominated for the observations. By means of the Wasserstein distance and a suitable sieve construction, our main result establishes PCRs in Bayesian nonparametric models where the posterior is available through a more general disintegration than the Bayes formula. To the best of our knowledge, this is the first general approach to provide PCRs in non-dominated Bayesian nonparametric models, and it relies on minimal modeling assumptions and on a suitable continuity assumption for the posterior distribution \citep{DM(20a),DM(20b),DFM(20)}. Some refinements of our result are presented under additional assumptions on the prior distribution, showing how the continuity assumption on the posterior distribution may be equivalently stated as an assumption on the predictive distributions induced by the prior. We apply our result in a Bayesian nonparametric framework with the Dirichlet process prior \citep{Fer(73)}, which is an example of a conjugate prior, and with the normalized extended Gamma process prior \citep[Example 2]{JLP(09)}, which is an example of a non-conjugate prior in the class of priors obtained by normalizing completely random measures \citep{JLP(09),Lij(10)}.

The paper is structured as follows. In Section \ref{sect:main_problem} we introduce Wasserstein PCRs in Bayesian nonparametrics, and present the key arguments of our approach to PCRs. In Section \ref{sec_mainres} we state and prove our main result on Wasserstein PCRs in non-dominated Bayesian nonparametric models, and then discuss some refinements of it under additional assumptions on the prior distribution. Section \ref{sect:illustrations} contains two examples of the application of our results, whereas in Section \ref{discuss} we discuss our work and also directions for future research. Auxiliary lemmas and proofs of complementary results are deferred to the Appendix.


\section{Wasserstein PCRs in Bayesian nonparametrics}\label{sect:main_problem}

We assume $\ss$-valued observations, with the space $\ss$ being a Polish space equipped with its Borel $\sigma$-field $\ssa$. Moreover, we denote by $\pms$ the space of all probability measures on $(\ss,\ssa)$, and we assume that $\pms$ is equipped with the corresponding Borel $\sigma$-field  $\Bcr(\pms)$ induced by the topology of weak convergence of probability measures \citep{GV(00)}. From a Bayesian perspective, observations are modeled as part of a sequence $X^{(\infty)} := \{X_i\}_{i \geq 1}$ of exchangeable random variables, each $X_{i}$'s taking values in $(\ss, \ssa)$ and defined on a common probability space $\probabilityspace$. By the de Finetti representation theorem, exchangeability of the observations is equivalent to the existence of a random probability measure $\randommeasure$ on $(\ss, \ssa)$ conditionally to which the $X_{i}$'s are independent and identically distributed, that is
\begin{displaymath}
\pp[X_1\in A_1, \dots, X_n\in A_n] = \int_{\pms } \prod_{i=1}^n p (A_i)  \pi (\ud p)
\end{displaymath}
for all $A_1, \ldots , A_n \in \ssa$ and $n \geq 1$. We have also denoted by
$\pi$ the probability distribution of $\randommeasure$, which is called the de Finetti measure of $X^{(\infty)}$. The core of Bayesian inferences is the posterior distribution, which is the conditional distribution of $\randommeasure$ given $(X_1, \dots, X_n)$, i.e. $\pp[\randommeasure \in \cdot|X_1, \dots, X_n]$.  The posterior can be represented by means of a probability kernel $\pi_n(\cdot| \cdot): \Bcr(\pms)\times\ss^n \rightarrow [0,1]$ satisfying the disintegration
\begin{equation} \label{disintegration}
\pp[X_1\in A_1, \dots, X_n\in A_n, \randommeasure\in B] = \int_{A_1\times\dots\times A_n} \pi_n(B|x^{(n)}) \mu_n(\ud x^{(n)})
\end{equation}
for all Borel sets $A_1, \dots, A_n \in \ssa$ and $B \in \Bcr(\pms)$ and $n\geq1$, where we set $x^{(n)} := (x_1, \dots, x_n)$ and
\begin{equation} \label{lawobservations}
\mu_n(A_1\times\dots\times A_n) := \pp (X_1 \in A_1, \ldots , X_n \in A_n), 
\end{equation}
so that $\pp[\randommeasure \in B|X_1, \dots, X_n] = \pi_n(B|X_1, \dots, X_n)$ is valid $\pp$\text{-a.s.} for any $B \in \Bcr(\pms)$. Another useful notion is that of predictive distribution, namely the conditional law 
$\pp[X_{n+1} \in \cdot|X_1, \dots, X_n]$. The predictive distribution can be represented by means of a probability kernel $\alpha_n(\cdot|\cdot) : \ssa \times \ss^n \rightarrow [0,1]$ for which
\begin{equation} \label{predictive}
\pp[X_1 \in A_1, \dots, X_n \in A_n, X_{n+1} \in A_{n+1}] = \int_{A_1 \times \dots \times A_n} \!\!\!\!\!\!\!\alpha_n(A_{n+1}|x^{(n)}) \mu_n(\ud x^{(n)})
\end{equation}
holds for all $A_1, \dots, A_n, A_{n+1} \in \ssa$, so that $\pp[X_{n+1} \in A_{n+1}|X_1, \dots, X_n] = \alpha_n(A_{n+1}|X_1, \dots, X_n)$. The disintegration \eqref{disintegration} and the predictive distribution \eqref{predictive} are critical for the development of our approach to PCRs.

\subsection{Wasserstein PCRs}

From \cite[Definition 6.1]{GV(00)} the posterior distribution is (weakly) consistent at $\pfrak_0\in\pms$ if, as $n\rightarrow+\infty$, the convergence $\pi_n(U_0^c| \xi_1, \dots, \xi_n) \rightarrow 0$ holds in probability for any neighborhood $U_0$ of $\pfrak_0$, where $\xi^{(\infty)} := \{\xi_i\}_{i \geq 1}$ stands for a sequence of $\ss$-valued independent random variables identically distributed as $ \pfrak_0$. The non uniqueness of the posterior distribution $\pi_n$ requires some additional regularity assumptions in order that the random measure $\pi_n(\cdot| \xi_1, \dots, \xi_n)$ is well-defined. The notion of PCR strengthens the notion of Bayesian consistency \cite[Chapter 8]{GV(00)}. In particular, a PCR allows to provide a precise quantification of Bayesian consistency. After the specification of a suitable distance $\ud_{\ps}$ on $\ps$, that yields $\Bcr(\ps)$ as relative Borel $\sigma$-algebra, the definition of PCR can be stated as follows \citep[Definition 8.1]{GV(00)}.

\begin{definition} \label{def:consistency}
A sequence $\{\epsilon_n\}_{n \geq 1}$ of (positive) numbers is defined to be a PCR at $\pfrak_0$ if, as $n\rightarrow+\infty$,
\begin{equation} \label{post_consistency}
\pi_n\left(\left\{\pfrak \in \ps : \ud_{\ps}(\pfrak,\pfrak_0) \geq M_n\epsilon_n\right\} \big| \xi_1, \dots, \xi_n \right) \rightarrow 0
\end{equation}
holds in probability for every choice of a sequence $\{M_n\}_{n \geq 1}$ of (positive) numbers such that $M_n \rightarrow +\infty$. 
\end{definition}
To summarize, a PCR quantifies the speed at which a $\ud_{\ps}$-neighborhood of the (true) parameter $\pfrak_0$ is allowed to shrink while maintaining, nevertheless, a very high posterior probability. We refer to \cite[Chpater 6-9]{GV(00)}, and references therein, for a comprehensive account on Bayesian consistency and PCRs.

Our approach to PCRs originates from a reformulation of Definition \ref{def:consistency} in terms of the $p$-Wasserstein distance \citep{AG,Vill(03)}. Let $(\Space, \ud_{\Space})$ be an (abstract) separable metric space, and denote by $\mathcal P(\Space)$ the relative space of all probability measures on $(\Space, \Bcr(\Space))$. Then, for any $p \geq 1$ the $p$-Wasserstein distance is defined as
\begin{equation} \label{Wasserstein}
\Wp^{(\mathcal P(\Space))}(\gamma_1; \gamma_2) := \inf_{\eta \in \mathcal{F}(\gamma_1,\gamma_2)} \left(\int_{\Space^2}  [\ud_{\Space}(x, y)]^p\ \eta(\ud x\ud y) \right)^{1/p}
\end{equation}
for any $\gamma_1, \gamma_2 \in \mathcal P_p(\Space)$, where 
\begin{displaymath}
\mathcal P_p(\Space):= \left\{\gamma \in \mathcal P(\Space)\ :\ \int_{\Space} [\ud_{\Space}(x, x_0)]^p \gamma(\ud x) < +\infty\, \ \text{for\ some\ } x_0 \in \Space\right\}
\end{displaymath}
and $\mathcal{F}(\gamma_1, \gamma_2)$ is the class of all probability measures on $(\Space^2, \Bcr(\Space^2))$ with $i$-the marginal $\gamma_i$, $i=1,2$. See \cite[Chapter 7]{AGS(08)} and, in particular, \cite[Proposition 7.1.5]{AGS(08)}. According to Definition \ref{def:consistency} and Lemma \ref{lem_pcr} the quantity
\begin{equation} \label{Wpcr}
\epsilon_n = \ee\left[ \Wp^{(\mathbb{P})}(\pi_n(\cdot| \xi_1, \dots, \xi_n); \delta_{\pfrak_0}) \right]
\end{equation}
gives a PCR at $\pfrak_0$, where $\delta_{\pfrak_0}$ denotes the degenerate distribution at $\pfrak_0$ and where we set $\mathbb{P}:= \mathcal P(\ps)$ to be the space of probability measures on $\ps$. We refer to $\epsilon_n $ in \eqref{Wpcr} as a $p$-Wasserstein PCR ($p$-WPCR).

Our main result provides a $p$-WPCR in non-dominated Bayesian nonparametric models. The key arguments of our approach may be summarized as follows. Following ideas developed in \cite{ReSa(00)}, we introduce a parameter $\delta > 0$ and a suitable finite partition $\{A_{j,\delta}\}_{j=1, \dots, N}$ of the metric space $(\ss, \ud_{\ss})$, that we assume to be totally bounded. Then, we show that such a partition induces a sequence of random variables $\{\eta_n\}_{n \geq 1}$, $\eta_n$ being an approximation of $\xi_n$, and a random probability measure $\randommeasure_{\delta}$, which is a discretized version of the directing (de Finetti) measure $\randommeasure$ of the sequence $X^{(\infty)}$ \citep{Ald(85)}. Then, we write
\begin{align} \label{splitPCR_NP}
\epsilon_n &= \ee\left[ \Wp^{(\ppms)}(\pi_n(\cdot| \xi_1, \dots, \xi_n); \delta_{\pfrak_0}) \right] \\
&\leq \ee\left[\Wp^{(\ppms)}(\pi_n(\cdot|\xi_1, \dots, \xi_n); \Gamma_N^{\ast}(\cdot|\eta_1, \dots, \eta_n)) \right] \nonumber\\
&\quad+ \ee\left[ \Wp^{(\ppms)}(\Gamma_N^{\ast}(\cdot|\eta_1, \dots, \eta_n) ;  \Sigma_N^{\ast}(\cdot|\eta_1, \dots, \eta_n) )\right] \nonumber \\
&\quad+ \ee\left[ \Wp^{(\ppms)}( \Sigma_N^{\ast}(\cdot|\eta_1, \dots, \eta_n); \delta_{\empiric^{(\eta)}} )\right] \nonumber \\
&\quad+ \ee\left[ \Wp^{(\pms)}(\empiric^{(\eta)}; \empiric^{(\xi)} ) \right] \nonumber  \\
&\quad+ \ee\left[ \Wp^{(\pms)}( \empiric^{(\xi)}; \pfrak_0) \right], \nonumber 
\end{align}
where $\empiric^{(\eta)} := n^{-1} \sum_{i=1}^n \delta_{\eta_i}$ is the empirical process, whereas $\Gamma_N^{\ast}$ and $\Sigma_N^{\ast}$ are the probability kernels that stand for the posterior distributions of $\randommeasure$ and $\randommeasure_{\delta}$, respectively, evaluated at the discretized (hypothetical) data $\eta_1, \dots, \eta_n$. On the right-most term of \eqref{splitPCR_NP} we observe the occurrence of $\ee[ \Wp^{(\mathcal P(\ss))}(\pfrak_0; \empiric^{(\xi)})]$, which is precisely the rate of convergence of a mean Glivenko-Cantelli theorem \citep{DR(19)}. According to \eqref{splitPCR_NP}, the problem of establishing a $p$-WPCR $\epsilon_n$ reduces to upper bound the five terms of the right-hand side of \eqref{splitPCR_NP}.


\section{Main results}\label{sec_mainres}

Before stating our main result, it is useful to give a schematic representation of the metric spaces, as wells as of their interplay, that appear within our Bayesian nonparametric framework under the $p$-Wasserstein distance. Links between these metric spaces are not established solely by the definition of the fundamental objects of the theory, such as the statistical model and the corresponding posterior distribution, but also by the $p$-Wasserstein distance, whose definition is strongly influenced by the base metric. In Figure \ref{figure1}, solid arrows point out the presence of a specific mapping. Furthermore, dotted arrows indicate that the Wasserstein distance corresponding to the head of the arrow is built on the metric space corresponding to the nock of the same arrow. Dashed arrows denote some metric construction, like the product or the quotient. In particular, the space $\ss^n/\!\!\sim$ stands for the quotient of the product space $\ss^n$ by the action of the symmetric group $\mathfrak S_n$. More precisely, upon specifying that we write
$$
\ud^n_{\ss}\big( (x_1, \dots, x_n); (y_1, \dots, y_n) \big) := \left(\frac 1n \sum_{i=1}^n [\ud_{\ss}(x_i, y_i)]^p \right)^{1/p}\ ,
$$
the quotient metric $\tilde{\ud_{\ss}^n}$ is
\begin{align*}
\tilde{\ud_{\ss}^n}([x], [y]) &:= \inf_{\substack{(x_1, \dots, x_n) \in [x] \\ (y_1, \dots, y_n) \in [y]}} \ud^n_{\ss}\big( (x_1, \dots, x_n); (y_1, \dots, y_n) \big) \nonumber \\
&= \inf_{\tau \in \mathfrak S_n}\ud^n_{\ss}\big( (x_1, \dots, x_n); (y_{\tau_n(1)}, \dots, y_{\tau_n(n)}) \big) 
\end{align*}
for all $[x], [y] \in \ss^n/\!\!\sim$. Now, a well-known theorem by Birkhoff \citep[Theorem 6.0.1]{AGS(08)} entails that $\tilde{\ud_{\ss}^n}([x], [y]) = \Wp^{(\pms)}(\empiric^{(x)}; \empiric^{(y)})$. Such a particular construction is critical in the proof of our main result.

\tikzstyle{block} = [ fill=white, 
    minimum height=3em, minimum width=6em]
\tikzstyle{line} = [-stealth, thick, draw]

\begin{figure}[!htb]
\begin{center}
\begin{tikzpicture}[node distance = 4.1cm, auto]
    \node [block] (fix) {$(\mathcal P_p(\ss), \Wp^{(\pms)})$};
    \node [block, above of=fix] (changed2) {$(\ss, \ud_{\ss})$};
    \node [block, below of=fix] (changed6) {$(\mathcal P_p(\ps), \Wp^{(\mathcal P(\ps)})$};
    \node [block, right of=changed2] (changed4) {$(\ss^n, \ud_{\ss}^n)$};
    \node [block, right of=fix] (changed5) {$(\ss^n/\!\!\sim, \tilde{\ud_{\ss}^n})$};
    \path [line,dotted] (changed2) -- (fix);
    \path [line] (fix) -- node {$\gamma \mapsto \pi_n^{\ast}(\cdot|\gamma)$} (changed6);
    \path [line] (changed4) -- node {$\mathfrak{e}_n^{(x)}$} (fix);
    \path [line,dashed] (changed2) -- node {product} (changed4);
    \path [line,dashed] (changed4) --  node {quotient}  (changed5);
    \path [line] (changed5) -- node {$x^{(n)} \mapsto \pi_n(\cdot|x^{(n)})$} (changed6);
\end{tikzpicture}
 \caption{The metrics spaces involved in the Bayesian nonparametric framework under the $p$-Wasserstein distance.}
 \end{center}
\label{figure1}
\end{figure}
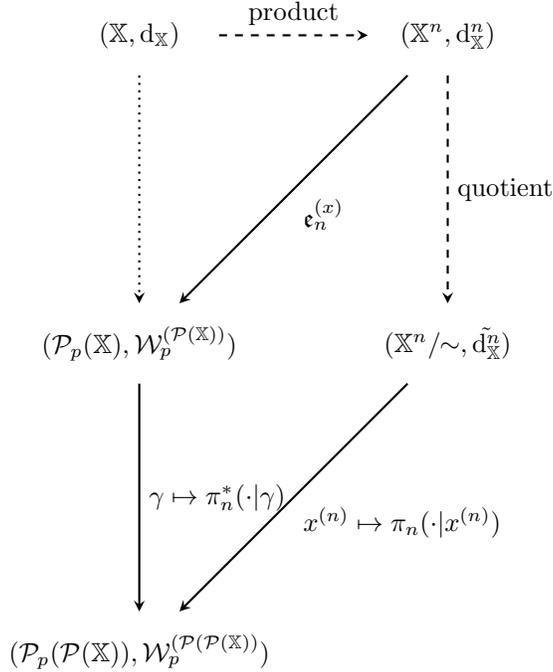

Now, we can state our main result on $p$-WPCR in non-dominated Bayesian nonparametric models. For a fixed $\delta > 0$, we consider the $\delta$-covering number $N_{\delta}(\ss, \ud_{\ss})$ of the metric space $(\ss, \ud_{\ss})$, which is supposed to be a totally bounded metric space.  Given such a  $\delta$-covering of $(\ss, \ud_{\ss})$, our starting point consists in finding a solution of the disintegration \eqref{disintegration}, which is denoted by $\pi_n^{\ast}(\cdot|\cdot)$, that satisfies the condition
\begin{equation} \label{post_continuity}
\Wp^{(\ppms)}(\pi_n^{\ast}(\cdot|x^{(n)});  \pi_n^{\ast}(\cdot|y^{(n)})) \leq L_n\ \Wp^{(\pms)}(\empiric^{(x)}; \empiric^{(y)})
\end{equation}
for all $x^{(n)} := (x_1, \dots,x_n)$ and $y^{(n)} := (y_1, \dots,y_n)$ in $\ss^n$, with some positive constant $L_n$.  Note that \eqref{post_continuity} may be regarded as an assumption of continuity for the posterior distribution in the sense of \cite{DFM(20)}. See \cite{DM(20a),DM(20b),DFM(20)} for details on continuity conditions for posterior distributions and their use in Bayesian consistency.

\begin{theorem} \label{main_thm2}
Assume that $(\ss, \ud_{\ss})$ is a totally bounded metric space, and that there exists a distinguished solution of the disintegration \eqref{disintegration} that fulfills \eqref{post_continuity} for all $x^{(n)}, y^{(n)} \in \ss^n$. 
Then, for any infinitesimal sequence $\delta_n$ of (positive) numbers, we can find two other infinitesimal sequences $M_n(N_{\delta_n}, \delta_n)$ and $V_n(N_{\delta_n}, \delta_n)$, depending only on the nonparametric 
prior through its finite-dimensional laws, such that
\begin{equation} \label{main_bound2}
\epsilon_n = \varepsilon_{n,p}(\ss, \pfrak_0) + 2(2 + L_n)\delta_n + \text{diam}(\ss) \left\{\frac 12 \left[M_n(N_{\delta_n}, \delta_n) + V_n(N_{\delta_n}, \delta_n)\right] \right\}^{1/p},
\end{equation}
with $ \varepsilon_{n,p}(\ss, \pfrak_0)$ being the rate of convergence of the mean Glivenko-Cantelli theorem, gives a $p$-WPCR at $\pfrak_0$. 
\end{theorem}

The definition of the functions $M_n$  and $V_n$ is given along the proof of Theorem \ref{main_thm2} in Section \ref{proof:thm2}. See Equation \eqref{MV} and the discussion thereafter. More precisely, the functions $M_n$  and $V_n$ are obtained from posterior means and posterior variances, respectively, of Bernoulli models with prior distributions of the form $\pp[\randommeasure(A) \in \cdot]$ and data coinciding with i.i.d. Bernoulli variables with common parameter $\kappa_0(A)$, for some subsets $A \in \ssa$. See also Section \ref{rmkk} for a discussion on the problem of upper bounding $M_{n}+V_{n}$. 

\subsection{Proof of Theorem \ref{main_thm2}}\label{proof:thm2}

For a fixed $\delta > 0$, put $N := N_{\delta}(\ss, \ud_{\ss})$ and indicate by $\{A_{j,\delta}\}_{j=1, \dots, N}$ a partition of $\ss$ such that $\text{diam}(A_{j,\delta}) \leq 2\delta$ and $\mu_1(\partial A_{j,\delta}) = 0$ 
for all $j \in \{1, \dots, N\}$, where $\mu_1$ is defined in \eqref{lawobservations}. Then, define $l_{\delta} : \ss \rightarrow \{1, \dots, N\}$ by the rule that, for any $x \in \ss$, $l_{\delta}(x) = j$ iff $x \in A_{j,\delta}$. 
Choose a point $a_{j,\delta} \in A_{j,\delta}$ for any $j \in \{1, \dots, N\}$, and put $\ss_{\delta} := \{a_{j,\delta}\}_{j=1, \dots, N}$ and $T_{\delta}(x) := a_{l_{\delta}(x),\delta}$, for any $x \in \ss$. 
Consider the new sequences $\{Y_i\}_{i \geq 1}$ and $\{\eta_i\}_{i \geq 1}$ of $\ss_{\delta}$-valued random variables, given by $Y_i := T_{\delta}(X_i)$ and $\eta_i := T_{\delta}(\xi_i)$ for any 
$i \in \naturals$, respectively. $\{Y_i\}_{i \geq 1}$ is a sequence of exchangeable random variables, with directing measure
$$
\randommeasure \circ T_{\delta}^{-1} = \sum_{j=1}^N \randommeasure(A_{j,\delta}) \delta_{a_{j,\delta}}\ .
$$
This follows from the mapping theorem for weak convergence, recalling that $\mu_1(\partial A_{j,\delta}) = 0$ for all $j \in \{1, \dots, N\}$.
Now, for any $n \in \naturals$, for any $y_1, \dots, y_n \in \ss_{\delta}$ and  for any $j \in \{1, \dots, N\}$, introduce the notation
$$
\nu_{\delta}(j; y_1, \dots, y_n) := \sum_{i=1}^n \ind\{y_i = a_{j,\delta}\},
$$
by which
$$
\pp[\randommeasure \in B\ |\ Y^{(n)}] = \frac{\int_B \prod_{j=1}^N [\pfrak(A_{j,\delta})]^{\nu_{\delta}(j; Y^{(n)})} \pi(\ud \pfrak)}{\int_{\pms} 
\prod_{j=1}^N [\pfrak(A_{j,\delta})]^{\nu_{\delta}(j; Y^{(n)})} \pi(\ud \pfrak)} \quad\quad \pp\text{-a.s.}
$$
holds for all $B \in \pmssa := \Bcr({\pms})$, where $Y^{(n)} := (Y_1, \dots, Y_n)$. This suggests to introduce the probability kernel
$$
\Gamma_N(B|\nu_1, \dots, \nu_N) := \frac{\int_B \prod_{j=1}^N [\pfrak(A_{j,\delta})]^{\nu_j} \pi(\ud \pfrak)}{\int_{\pms} \prod_{j=1}^N [\pfrak(A_{j,\delta})]^{\nu_j} \pi(\ud \pfrak)}
$$
for $(B; \nu_1, \dots, \nu_N) \in \pmssa \times \naturals_0^N$, where $\naturals_0 := \{0\} \cup \naturals$, and define the corresponding random probability measure
$$
\Gamma_N^{\ast}(B|\eta_1, \dots, \eta_n) :=  \Gamma_N(B|\nu_{\delta}(1; \eta^{(n)}), \dots, \nu_{\delta}(N; \eta^{(n)}))
$$
for $B \in \pmssa$, where $\eta^{(n)} := (\eta_1, \dots, \eta_n)$. Consider the finite-dimensional distribution of $\pi$ relative to $\{A_{j,\delta}\}_{j=1, \dots, N}$, i.e. $\pi_N(D; A_{1,\delta} \dots A_{N,\delta}) := \pp\left[\left( \randommeasure(A_{1,\delta}), \dots, \randommeasure(A_{N-1,\delta}) \right) \in D \right]$ for any $D \in \Bcr(\Delta_{N-1})$, and set
\begin{align*}
&\Upsilon_N(D\ |\ \nu_1, \dots, \nu_N):= \frac{ \int_D \left[ \prod_{j=1}^{N-1} u_j^{\nu_j} \right] \left(1 - \sum_{j=1}^{N-1} u_j \right)^{\nu_N} \pi_N(\ud\ub; A_{1,\delta} \dots A_{N,\delta})}
{\int_{\Delta_{N-1}} \left[ \prod_{j=1}^{N-1} u_j^{\nu_j} \right] \left(1 - \sum_{j=1}^{N-1} u_j \right)^{\nu_N} \pi_N(\ud\ub; A_{1,\delta} \dots A_{N,\delta})}
\end{align*}
for any $(D; \nu_1, \dots, \nu_N) \in \mathscr{B}(\Delta_{N-1}) \times \naturals_0^N$. Denoting by $G_{\delta} : \Delta_{N-1} \rightarrow \mathcal{P}(\ss_{\delta})$ the one-to-one mapping which sends $\ub \in \Delta_{N-1}$ into the probability measure $\sum_{j=1}^N u_j \delta_{a_{j,\delta}}$, where $u_N := 1 - \sum_{j=1}^{N-1} u_j$, observe that
\begin{align*}
&\pp[\randommeasure \circ T_{\delta}^{-1} \in F|Y^{(n)}]= \Upsilon_N(G_{\delta}^{-1}(F)|\nu_{\delta}(1; Y^{(n)}), \dots, \nu_{\delta}(N; Y^{(n)}))  \quad\quad \pp\text{-a.s.}
\end{align*}
holds for any $F \in \Bcr(\mathcal{P}(\ss_{\delta}))$. Now, if $\iota_{\delta} : \ss_{\delta} \rightarrow \ss$ indicates the (canonical) inclusion map, then we observe that $I_{\delta} : \pfrak \mapsto \pfrak \circ \iota_{\delta}^{-1}$ induces the inclusion of the space $\mathcal{P}(\ss_{\delta})$ into the space $\pms$. Lastly, we set
$$
\Sigma_N(B|\nu_1, \dots, \nu_N) := \Upsilon_N(G_{\delta}^{-1}(I_{\delta}^{-1}(B))|\nu_1, \dots, \nu_N) 
$$
for any $(B; \nu_1, \dots, \nu_N) \in \pmssa \times \naturals_0^N$ and 
$$
\Sigma_N^{\ast}(B|\eta^{(n)}) :=  \Sigma_N(B|\nu_{\delta}(1; \eta^{(n)}), \dots, \nu_{\delta}(N; \eta^{(n)}))\ . 
$$
Upon noticing that 
$$
\Wp^{(\ppms)}(\delta_{\empiric^{(\eta)}}; \delta_{\empiric^{(\xi)}}) = \Wp^{(\pms)}(\empiric^{(\eta)}; \empiric^{(\xi)}) 
$$
and
$$
\Wp^{(\ppms)}(\delta_{\empiric^{(\xi)}}; \delta_{\pfrak_0}) = \Wp^{(\pms)}( \empiric^{(\xi)}; \pfrak_0) \ , 
$$
there are now all the elements to deal with the inequality displayed in \eqref{splitPCR_NP}. In the remaining part of the proof se show how to manipulate each of the five terms on the right-hand side of \eqref{splitPCR_NP} in oder to get the $p$-WPCR \eqref{main_bound2}.

With regards to the first term $\ee[\Wp^{(\ppms)}(\pi_n(\cdot|\xi_1, \dots, \xi_n); \Gamma_N^{\ast}(\cdot|\eta_1, \dots, \eta_n))] $ on the right-hand side of \eqref{splitPCR_NP}, note that
$$
\pp[\randommeasure \in B\ |\ Y^{(n)}] = \frac{\int_{C^{(n)}[Y^{(n)}]} \pi_n(B| x^{(n)}) \mu_n(\ud x^{(n)})}{\mu_n(C^{(n)}[Y^{(n)}])}  \quad\quad \pp\text{-a.s.}
$$
is a valid probability measure by de Finetti's representation theorem, where $C^{(n)}[Y^{(n)}]$ is defined as follows
$$
C^{(n)}[Y^{(n)}] = \times_{j=1}^N  A_{j,\delta}^{\nu_{\delta}(j; Y^{(n)})} := \underbrace{A_{1,\delta} \times \dots \times A_{1,\delta}}_{\nu_{\delta}(1; Y^{(n)})\text{-times}} \times \dots \times
\underbrace{A_{N,\delta} \times \dots \times A_{N,\delta}}_{\nu_{\delta}(N; Y^{(n)})\text{-times}}\ . 
$$
Whence, 
$$
\Gamma_N^{\ast}(B|\eta_1, \dots, \eta_n) 
= \frac{\int_{C^{(n)}[\eta^{(n)}]} \pi_n(B|x^{(n)}) \mu_n(\ud x^{(n)})}{\mu_n(C^{(n)}[\eta^{(n)}])} \quad\quad \pp\text{-a.s.}
$$
and, by convexity of $\Wp^p$ \citep[Chapter 7]{Vill(03)}, 
\begin{align*}
& \left[\Wp^{(\ppms)}(\pi_n(\cdot|\xi_1, \dots, \xi_n); \Gamma_N^{\ast}(\cdot|\eta_1, \dots, \eta_n))\right]^p \\
& \quad\leq \frac{\int_{C^{(n)}[\eta^{(n)}]} \left[\Wp^{(\ppms)}\left(\pi_n(\cdot|x^{(n)}); \pi_n(\cdot|\xi^{(n)})\right)\right]^p \mu_n(\ud x^{(n)})}{\mu_n(C^{(n)}[\eta^{(n)}])}\ .
\end{align*}
At this stage, observe that an application of the (continuity) condition displayed in \eqref{post_continuity} yields the inequality
$$
\Wp^{(\ppms)}(\pi_n(\cdot|x^{(n)}); \pi_n(\cdot\ |\ \xi^{(n)})) \leq L_n\ \Wp^{(\pms)}(\empiric^{(x)}; \empiric^{(\xi)}) \quad\quad \pp\text{-a.s.} 
$$
for any $x^{(n)} \in C^{(n)}[\eta^{(n)}]$. In particular, the definition of the relation $x^{(n)} \in C^{(n)}[\eta^{(n)}]$ entails $\sum_{i=1}^n \ind\{x_i \in A_{j,\delta}\} = \sum_{i=1}^n \ind\{\xi_i \in A_{j,\delta}\}$ for all $j \in \{1, \dots, N\}$. Therefore, a direct application of the Birkhoff theorem \citep[Theorem 6.0.1]{AGS(08)} shows that $\Wp^{(\pms)}(\empiric^{(x)}; \empiric^{(\xi)}) \leq 2\delta$, $\pp$\text{-a.s.}.
Whence, we write
\begin{align*}
\left[\Wp^{(\ppms)}(\pi_n(\cdot|\xi_1, \dots, \xi_n); \Gamma_N^{\ast}(\cdot|\eta_1, \dots, \eta_n))\right]^p &\leq \frac{\int_{C^{(n)}[\eta^{(n)}]} (2L_n\delta)^p \mu_n(\ud x^{(n)})}{\mu_n(C^{(n)}[\eta^{(n)}])} \\
&= (2L_n\delta)^p \qquad \pp\text{-a.s.}.
\end{align*}

With regards to the second term $\ee[ \Wp^{(\ppms)}(\Gamma_N^{\ast}(\cdot|\eta_1, \dots, \eta_n) ;  \Sigma_N^{\ast}(\cdot|\eta_1, \dots, \eta_n) )]$ on the right-hand side of \eqref{splitPCR_NP}, we start by observing that for any bounded and measurable function $h : \pms \rightarrow \reals$ we can write that
$$
\int_{\pms} h(\pfrak) \Sigma_N(\ud\pfrak\ |\ \nu_1, \dots, \nu_N) = \int_{\pms} h(\pfrak \circ T_{\delta}^{-1}) \Gamma_N(\ud\pfrak\ |\ \nu_1, \dots, \nu_N),
$$ 
which holds for all $(\nu_1, \dots, \nu_N) \in \naturals_0^N$. Accordingly, there exists a distinguished coupling between the probability kernel $\Sigma_N(\cdot\ |\ \nu_1, \dots, \nu_N)$ and the probability kernel $\Gamma_N(\cdot\ |\ \nu_1, \dots, \nu_N)$ that yields the inequality
\begin{align*}
&\Wp^{(\ppms)}(\Sigma_N(\cdot\ |\ \nu_1, \dots, \nu_N);  \Gamma_N(\cdot\ |\ \nu_1, \dots, \nu_N)) \\
&\quad\leq \left(\int_{\pms} \left[\Wp^{(\pms)}(\pfrak; \pfrak \circ T_{\delta}^{-1})\right]^p \Gamma_N(\ud\pfrak\ |\ \nu_1, \dots, \nu_N)\right)^{1/p}
\end{align*}
for all $(\nu_1, \dots, \nu_N) \in \naturals_0^N$. At this stage, again by choosing a distinguished coupling, the following inequality holds 
$$
\left[\Wp(\pfrak; \pfrak \circ T_{\delta}^{-1})\right]^p \leq \int_{\ss} [\ud_{\ss}(x, T_{\delta}(x))]^p \pfrak(\ud x)\ .
$$
Moreover, by a direct application the law of total probability, it is easy to show that the above right-hand side is equal to 
$$
\sum_{j=1}^N \pfrak(A_{j,\delta}) \int_{\ss} [\ud_{\ss}(x, T_{\delta}(x))]^p \pfrak(\ud x|A_{j,\delta})\ , 
$$
where $\pfrak(\cdot|A_{j,\delta})$ is the probability measure $A \mapsto \pfrak(A \cap A_{j,\delta})/\pfrak(A_{j,\delta})$ in the case that $\pfrak(A_{j,\delta}) > 0$ and any other probability measure (e.g., $\pfrak_0(\cdot)$) in the case that $\pfrak(A_{j,\delta}) = 0$. As far as those $A_{j,\delta}$'s for which $\pfrak(A_{j,\delta}) > 0$, we get
$$
\int_{\ss} [\ud_{\ss}(x, T_{\delta}(x))]^p \pfrak(\ud x|A_{j,\delta}) \leq (2\delta)^p
$$ 
for all $j \in \{1, \dots, N\}$, because $\pfrak(\cdot|A_{j,\delta})$ is supported in $A_{j,\delta}$. Therefore, we conclude that the second term \\ $[ \Wp^{(\ppms)}(\Gamma_N^{\ast}(\cdot|\eta_1, \dots, \eta_n) ;  \Sigma_N^{\ast}(\cdot|\eta_1, \dots, \eta_n) )]$ on the right-hand side of \eqref{splitPCR_NP} is globally bounded by $2\delta$. 

With regards to the third term $\ee[ \Wp^{(\ppms)}( \Sigma_N^{\ast}(\cdot|\eta_1, \dots, \eta_n); \delta_{\empiric^{(\eta)}} )]$ on the right-hand side of \eqref{splitPCR_NP}, because of the definition of the probability kernel $\Sigma_N(\cdot|\nu_1, \dots, \nu_N)$, we can write the argument of the expectation as \begin{align*}
&\Wp^{(\ppms)}\left(\delta_{\empiric^{(\eta)}}; \Sigma_N(\cdot\ |\ \nu_{\delta}(1; \eta^{(n)}), \dots, \nu_{\delta}(N; \eta^{(n)})) \right) \\
&\quad= \left( \int_{\pms} \left[ \Wp^{(\pms)}(\empiric^{(\eta)}; \pfrak) \right]^p \Sigma_N\big( \ud\pfrak\ |\ \nu_{\delta}(1; \eta^{(n)}), \dots, \nu_{\delta}(N; \eta^{(n)}) \big) \right)^{1/p}\\
&\quad= \left( \int_{\Delta_{N-1}} \left[ \Wp^{(\pms)}\Big( \empiric^{(\eta)}; \sum_{j=1}^N u_j \delta_{a_{j,\delta}} \Big)\right]^p \Upsilon_N\big( \ud\ub\ |\ \nu_{\delta}(1; \eta^{(n)}), \dots, \nu_{\delta}(N; \eta^{(n)}) \big)\right)^{1/p}\ .
\end{align*}
Now, observe that 
$$
\empiric^{(\eta)} = \frac 1n \sum_{i=1}^n \sum_{j=1}^N \ind\{\eta_i = a_{j,\delta}\} \delta_{a_{j,\delta}} = \sum_{j=1}^N \phi_{j,n}^{(\xi)} \delta_{a_{j,\delta}} \ , 
$$
where $\phi_{j,n}^{(\xi)} := \frac 1n \sum_{i=1}^n \ind\{\xi_i \in A_{j,\delta}\}$. Therefore, upon noticing that $\empiric^{(\eta)}$ and $\sum_{j=1}^N u_j \delta_{a_{j,\delta}}$ are both supported on $\ss_{\delta}$, the fact the total variation distance is given by an optimal coupling \citep[page 424]{GiSu(07)} yields
$$
\left[\Wp^{(\pms)}\Big( \empiric^{(\eta)}; \sum_{j=1}^N u_j \delta_{a_{j,\delta}} \Big)\right]^p \leq \frac 12 [\text{diam}(\ss)]^p \sum_{j=1}^N \big| \phi_{j,n}^{(\xi)} - u_j \big| \ . 
$$
Therefore, $\pp$\text{-a.s.}, it holds
\begin{align*}
& \Wp^{(\ppms)} \left(\delta_{\empiric^{(\eta)}}; \Sigma_N(\cdot\ |\ \nu_{\delta}(1; \eta^{(n)}), \dots, \nu_{\delta}(N; \eta^{(n)})) \right) \\
&\quad \leq \text{diam}(\ss) \left( \frac 12 \int_{\Delta_{N-1}} \!\!\! \left[ \sum_{j=1}^N \big| \phi_{j,n}^{(\xi)} - u_j \big|\right] \Upsilon_N\big( \ud\ub|\nu_{\delta}(1; \eta^{(n)}), \dots, \nu_{\delta}(N; \eta^{(n)}) \big)\right)^{1/p}.
\end{align*}
In particular, we observe that the previous upper bound suggests that is worth studying the following integral
\begin{align*}
& \int_0^1 |\phi_{j,n}^{(\xi)} - u_j | \Upsilon_{N,j}\big( \ud u_j|\nu_{\delta}(1; \eta^{(n)}), \dots, \nu_{\delta}(N; \eta^{(n)}):= \frac{\int_0^1 |\phi_{j,n}^{(\xi)} - t| t^{n \phi_{j,n}^{(\xi)}} (1-t)^{n(1-\phi_{j,n}^{(\xi)})} \pi_{N,j}(\ud t)}{\int_0^1 t^{n \phi_{j,n}^{(\xi)}} (1-t)^{n(1-\phi_{j,n}^{(\xi)})}  \pi_{N,j}(\ud t)}
\end{align*}
where $\Upsilon_{N,j}\big( \cdot|\nu_1, \dots, \nu_N)$ denotes the $j$-th marginal of $\Upsilon_{N}\big( \cdot|\nu_1, \dots, \nu_N)$, and $\pi_{N,j}(\cdot) = \pp[\randommeasure(A_{j,\delta}) \in \cdot]$. Upon setting
\begin{equation*} \label{post_mean_finite_dim}
m_{j,n}^{(\xi)}(\pi_{N,j}) := \frac{\int_0^1 t^{1+ n \phi_{j,n}^{(\xi)}} (1-t)^{n(1-\phi_{j,n}^{(\xi)})} \pi_{N,j}(\ud t)}{\int_0^1 t^{n \phi_{j,n}^{(\xi)}} (1-t)^{n(1-\phi_{j,n}^{(\xi)})}  \pi_{N,j}(\ud t)}
\end{equation*}
and 
\begin{equation*} \label{post_var_finite_dim}
v_{j,n}^{(\xi)}(\pi_{N,j}) := \frac{\int_0^1 [t - m_{j,n}^{(\xi)}(\pi_{N,j})]^2 t^{n \phi_{j,n}^{(\xi)}} (1-t)^{n(1-\phi_{j,n}^{(\xi)})} \pi_{N,j}(\ud t)}{\int_0^1 t^{n \phi_{j,n}^{(\xi)}} (1-t)^{n(1-\phi_{j,n}^{(\xi)})}  \pi_{N,j}(\ud t)}\ ,
\end{equation*}
it holds
\begin{align}
& \ee\left[\Wp^{(\ppms)} \left(\delta_{\empiric^{(\eta)}}; \Sigma_N(\cdot\ |\ \nu_{\delta}(1; \eta^{(n)}), \dots, \nu_{\delta}(N; \eta^{(n)})) \right)\right] \nonumber \\
&\quad\leq \text{diam}(\ss) \left\{ \frac 12 \sum_{j=1}^N \ee\left[ |\phi_{j,n}^{(\xi)} - m_{j,n}^{(\xi)}(\pi_{N,j}) | +\sqrt{v_{j,n}^{(\xi)}(\pi_{N,j})} \right] \right\}^{1/p} \nonumber \\
&\quad =: \text{diam}(\ss) \left\{ \frac 12 \left[ M_n(N,\delta) + V_n(N,\delta) \right] \right\}^{1/p} \label{MV},
\end{align}
where we set $M_n(N,\delta):=\sum_{j=1}^{N}\ee[|\phi_{j,n}^{(\xi)} - m_{j,n}^{(\xi)}(\pi_{N,j})|]$ and $V_n(N,\delta):=\sum_{j=1}^{N}\ee[\sqrt{v_{j,n}^{(\xi)}(\pi_{N,j})}]$
To conclude the proof, the asymptotic evaluation of the terms $|\phi_{j,n}^{(\xi)} - m_{j,n}^{(\xi)}(\pi_{N,j})|$ and $v_{j,n}^{(\xi)}(\pi_{N,j})$ is a longstanding topic in classical probability theory and Bayesian statistics. In particular, precise asymptotic expansions of these terms can be found in \cite[Section 3]{Johnson(70)} and \cite[Chapter 20]{DasGupta(08)}, where it is shown that, under suitable regularity assumptions on $\pi_{N,j}$, both $|\phi_{j,n}^{(\xi)} - m_{j,n}^{(\xi)}(\pi_{N,j})|$ and $v_{j,n}^{(\xi)}(\pi_{N,j})$ are of order $n^{-1}$. 

With regards to the fourth term $\ee[ \Wp^{(\pms)}(\empiric^{(\eta)}; \empiric^{(\xi)} )]$ on the right-hand side of \eqref{splitPCR_NP}, such a term contains $\Wp^{(\pms)}\big(\empiric^{(\xi)}; \empiric^{(\eta)}\big)$, which can be bounded in view of the Birkhoff theorem \citep[Theorem 6.0.1]{AGS(08)} as follows
\begin{displaymath}
\Wp^{(\pms)}\big(\empiric^{(\xi)}; \empiric^{(\eta)}\big) \leq \left( \frac1n \sum_{i=1}^n [\ud_{\ss}(\xi_i; \eta_i)]^p \right)^{1/p}\ .
\end{displaymath}
From the definition of $\eta_i$, $\ud_{\ss}(\xi_i; \eta_i) \leq \max_{j=1, \dots, N} \text{diam}(A_{j,\delta}) \leq 2\delta$. Then $\Wp^{(\pms)}\big(\empiric^{(\xi)}; \empiric^{(\eta)}\big) \leq 2\delta$ holds $\pp\text{-a.s.}$. 

With regards to the fifth term $\ee[\Wp^{(\pms)}( \empiric^{(\xi)}; \pfrak_0)] $ on the right-hand side of \eqref{splitPCR_NP}, we observe that such a term is precisely the rate of convergence of a mean Glivenko-Cantelli theorem \citep{DR(19)}. This completes the proof.

\subsection{Some refinements of Theorem \ref{main_thm2}}\label{rmkk}

Theorem \ref{main_thm2} provides the first general approach to deal with PRCs in non-dominated Bayesian nonparametric models. In particular, Theorem \ref{main_thm2} is valid under minimal modeling assumptions, which makes the resulting PCR not explicit and possibly not sharp. Additional assumptions may be considered in order to improve Theorem \ref{main_thm2}, and hereafter we present a few examples that go in such a direction. First, we apply Theorem \ref{main_thm2} to obtain a more explicit $p$-WPCRs, whenever the metric structure of the space $(\ss, \ud_{\ss})$ is known along with the system of finite-dimensional distributions of the prior. Under additional assumptions, the next corollary quantifies the order of decaying of both $M_n(N_{\delta_n}, \delta_n)$ and $V_n(N_{\delta_n}, \delta_n)$, thus making the $p$-WPCR \eqref{main_bound2} more explicit. See Appendix \ref{proof_cor} for the proof of Corollary \ref{main_cor2}. 

\begin{corollary} \label{main_cor2}
Under the framework of Theorem \ref{main_thm2}, we consider the following additional assumptions:
\begin{enumerate}
\item[i)] $N_{\delta}(\ss, \ud_{\ss}) \sim (1/\delta)^{d}$ for some $d > 0$, as $\delta \to 0$;
\item[ii)] $L_n \sim n^s$ for some $s \geq 0$, as $n \to +\infty$;
\item[iii)] for any $\delta > 0$, there holds
\begin{equation} \label{asymptoticMV}
M_n(N_{\delta}, \delta) + V_n(N_{\delta}, \delta) \leq C(\pi) N n^{-\alpha}
\end{equation}
for some positive constant $C(\pi)$ depending only on the prior $\pi$ and $\alpha > sd$. 
\end{enumerate}
Then, the $p$-WPCR \eqref{main_bound2} has an upper bound whose asymptotic expansion, as $ n\rightarrow +\infty$, is $\varepsilon_{n,p}(\ss, \pfrak_0) + n^{-(\alpha - ds)/(d + p)}$.
\end{corollary}
The problem of verifying the assumption iii) of Corollary \ref{main_cor2}, which provides an estimate of the critical term $M_n(N_{\delta}, \delta) + V_n(N_{\delta}, \delta)$, is straightforward under the assumption that the prior distribution $\pi$ has $1$-dimensional distributions, i.e. $\pp[\randommeasure(A) \in \cdot]$ for $A \in \ssa$, that coincide with Beta distributions. This is well-known to be the case of the Dirichlet process prior \citep{Fer(73)}. See also Appendix \ref{ineq} for a more general estimate of the term $M_n(N_{\delta}, \delta) + V_n(N_{\delta}, \delta)$, which is obtained by relying on a concentration inequality in \cite{DF(90)}.

We conclude our study on $p$-WPCR by presenting a proposition which is useful for the application of Theorem \ref{main_thm2} under the assumption that the posterior distribution in known explicitly. Again, because of its conjugacy, the Dirichlet process prior \citep{Fer(73)} is arguably the most notable example of a prior that satisfies such an assumption. Other examples are in the broad class of priors obtained by normalizing completely random measures \citep{JLP(09)}. In particular, the next proposition shows the critical role of the predictive distribution \eqref{predictive} in establishing $p$-WPCR through Theorem \ref{main_thm2}. See Appendix \ref{proof_prp2} for the proof of Proposition \ref{main_prop2}.

\begin{proposition} \label{main_prop2}
For any fixed $n \in \naturals$, the assumption \eqref{post_continuity} for $p=1$ and $L_n = \bar{L}$ for every $n \in \naturals$
is equivalent to
\begin{equation} \label{Lipschitz_predictive}
\Wuno^{(\pms)}(\alpha_n^{\ast}(\cdot|x^{(n)});  \alpha_n^{\ast}(\cdot|y^{(n)})) \leq \bar{L}\ \Wuno^{(\pms)}(\empiric^{(x)}; \empiric^{(y)})
\end{equation}
for all $x^{(n)}, y^{(n)} \in \ss^n$, with $\alpha_n^{\ast}$ being a distinguished solution of the predictive distribution displayed in \eqref{predictive}.
\end{proposition}


\section{Examples}\label{sect:illustrations}

We apply Theorem \ref{main_thm2} in a Bayesian nonparametric framework with the popolar Dirichlet process prior \citep{Fer(73)} and with the normalized extended Gamma process prior \citep[Example 2]{JLP(09)}. The Dirichlet process prior is a conjugate prior; we show that it satisfies both the assumptions of Corollary \ref{main_cor2} and the assumption \eqref{Lipschitz_predictive} of Proposition \ref{main_prop2}, and hence Theorem \ref{main_thm2} holds. The normalized extended Gamma process prior is a non-conjugate prior; we show that it satisfies the assumption \eqref{Lipschitz_predictive} of Proposition \ref{main_prop2}, and hence Theorem \ref{main_thm2} holds.

\subsection{The Dirichlet process prior}

We consider the Dirichlet process prior on the space $(\ss, \ssa)$. Let $q > 0$ and $H \in \pms$. It is well-known that the predictive distribution of a Dirichlet process with total mass $q$ and mean p.m. $H$ is given by
$$
\alpha_n(\cdot\ |\ x_1, \dots, x_n) = \frac{q}{q+n} H(\cdot) + \frac{n}{q+n} \empiric^{(x)}\ .
$$
Whence,
$$
\Wuno^{(\pms)}\left(\alpha_n(\cdot\ |\ x_1, \dots, x_n); \alpha_n(\cdot\ |\ y_1, \dots, y_n)\right) = \frac{n}{q+n} \Wuno^{(\pms)}(\empiric^{(x)}; \empiric^{(y)})
$$
for all $n \in \naturals$ and $x^{(n)}, y^{(n)} \in \ss^n$, entailing that Proposition \ref{main_prop2} is directly applicable, with $\bar{L}=1$. Then, Theorem \ref{main_thm2} holds. According to the definition of the Dirichlet process prior in terms of its finite-dimensional distributions \citep{Fer(73),Reg(01)}, Corollary \ref{main_cor2} can also be applied. In particular, according to Equation \eqref{MV} we have that
\begin{align*}
m_{j,n}^{(\xi)}(\pi_{N,j}) &= \frac{qH(A_{j,\delta}) + n \phi_{j,n}^{(\xi)}}{q + n}
\end{align*}
and
\begin{align*}
v_{j,n}^{(\xi)}(\pi_{N,j}) &= \frac{[qH(A_{j,\delta}) + n \phi_{j,n}^{(\xi)}] \cdot [qH(A^c_{j,\delta}) + n(1- \phi_{j,n}^{(\xi)})]}{(q + n)^2(q+n+1)}
\end{align*}
for $j = 1, \dots, N$. Whence,
$$
|m_{j,n}^{(\xi)}(\pi_{N,j}) - \phi_{j,n}^{(\xi)}| = \frac{q}{q+n}|H(A_{j,\delta}) - \phi_{j,n}^{(\xi)}| \leq \frac{q}{q+n}
$$
yielding that $M_n(N_\delta, \delta) \leq n^{-1}qN_{\delta}$. On the other hand, it is straightforward to show that $v_{j,n}^{(\xi)}(\pi_{N,j}) \leq n^{-1}$, so that it holds $V_n(N_\delta, \delta) \leq n^{-1/2}N_{\delta}$. In conclusion, the assumption \eqref{asymptoticMV} assumes the following form
$$
M_n(N_\delta, \delta) + V_n(N_\delta, \delta) \leq \frac{(q+1)N_{\delta}}{\sqrt{n}} \ .
$$
That is, Corollary \ref{main_cor2} holds true for some choice of $d > 0$, and for $s=0$ and $\alpha = 1/2$. Then, Theorem \ref{main_thm2} holds.

\subsection{The normalized extended Gamma process prior}

We consider the normalized extended Gamma process prior on the space $(\ss, \ssa)$, which is an example of a prior obtained by normalizing completely random measures \citep{JLP(09)}. See also \cite{Lij(10)} and \cite[Chapter 3]{GV(00)} for details. If $\tilde{\mu}$ denotes a completely random measure on $(\ss,\ssa)$ with L\'evy intensity measure $\nu (\ud s, \ud x) = s^{-1}e^{-s \beta (x)} \ud s  \alpha (\ud x)$ on $\reals^+ \times \ss$, where $\alpha$ is a finite measure on $(\ss, \ssa)$ with total mass $0 < a < +\infty$, then the normalized extended Gamma process prior is defined as the distribution of $\randommeasure := \tilde{\mu}/\tilde{\mu} (\ss)$. Here, we assume that $\beta : \ss  \to \reals^+$ is a Lipschitz function with Lipschitz constant $L >0$ such that there exist two constants $\beta_0$ and $\beta_1$ with the property $0 < \beta_0 \leq  \beta (x) \leq \beta_1 < +\infty  $ for any $x \in \ss$. According to  \cite[Proposition 2]{JLP(09)}, the predictive distribution of the normalized extended Gamma process equals
\begin{equation} \label{eq:predictive_CRM}
\pp (X_{n+1} \in \ud z | X_1, \ldots , X_n  ) =  w^{(n)} (z) \alpha (\ud z )  + \frac{1}{n} \sum_{j=1}^{K_n} w_j^{(n)} \delta_{X_j^*} (\ud z )
\end{equation} 
having denoted by $X_1^* ,\ldots , X_{K_n}^*$ the $K_n \leq n$ distinct values out of the sample $(X_1, \ldots , X_n)$, and where
\begin{displaymath}
w^{(n)} (z) = \frac{1}{n}\int_0^{+\infty}  \frac{u }{u +\beta (z)} f_{U_n}^{(x)}  (u ) \ud u
\end{displaymath}
and
\begin{displaymath}
w^{(n)}_j = n_j \int_0^{+\infty}  \frac{u }{u +\beta (X_j^*)} f_{U_n}^{(x)} (u) \ud u
\end{displaymath}
with $f_{U_n}^{(x) }$ being a density function on $\reals^+$ depending on the observed sample $(X_1, \ldots , X_n)$ and defined as
\begin{equation}\label{eq:U_n}
f_{U_n}^{(x)} (u)= \frac{u^{n-1} \exp \left\{ -\int_\ss  \log (u+\beta (z))[\alpha + n \mathfrak{e}_n^{(x)}](\ud z) \right\}}{
\int_0^{+\infty} u^{n-1} \exp \left\{ -\int_\ss  \log (u+\beta (z))[\alpha + n \mathfrak{e}_n^{(x)}](\ud z) \right\} \ud u}.
\end{equation}
By combining \eqref{eq:predictive_CRM} with \eqref{eq:U_n} through $w^{(n)} (z) $ and $w^{(n)}_j $, we can write the predictive distribution as follows
\begin{equation}\label{eq:predictive_gamma}
\pp (X_{n+1} \in \ud z | X_1, \ldots , X_n) =  \frac{1}{n} 	\int_0^{+\infty}  \frac{u }{u +\beta (z) }  f_{U_n}^{(x)}  (u) \ud u [\alpha+ n \mathfrak{e}_n^{(x)} ] (\ud z), 
\end{equation}
and we set $\zeta_n^{(x)} (z) :=n^{-1}\int_0^{+\infty}u (u +\beta (z))^{-1} f_{U_n}^{(x)}  (u) \ud u$ for easy of notation. Now, we show that the normalized extended Gamma process prior satisfies the assumption \eqref{Lipschitz_predictive} of Proposition \ref{main_prop2}. Then, we consider the Wasserstein distance between predictive distributions referring to different initial observed samples of size $n$, say $(X_1, \ldots , X_n) = (x_1, \ldots , x_n)$ and $(Y_1, \ldots , Y_n ) = (y_1, \ldots , y_n)$, respectively, that is
\begin{displaymath}
\Wuno^{(\Pc (\ss))} (\alpha_n (\, \cdot \, | x^{(n)}) ; \alpha_n (\, \cdot \, | y^{(n)}))
\end{displaymath}
where $\alpha_n (\, \cdot \, | x^{(n)})$ and $\alpha_n (\, \cdot \, | y^{(n)})$ indicate the predictive distributions as in \eqref{eq:predictive_gamma} for the two observed samples  $x^{(n)}= (x_1, \ldots , x_n)$ and $y^{(n)}= (y_1, \ldots , y_n)$ in $\ss^n$.
For the sake of simplifying notation we denote by $\text{Lip}_1$ the space of all Lipschitz functions $g: \ss \to \reals^+$ with Lipschitz constant less or equal then $1$, i.e., 
\begin{displaymath}
\text{Lip}_1: = \left\{ g: \ss \to \reals^+ : \; \exists \, L \leq 1 \text{ with } |g(x)-g(y)| \leq L \ud_\ss (x,y) , \text{ for any } x, y \in \ss\right\}.
\end{displaymath}
Then,
\begin{align}\label{eq:threeterms} 
&\Wuno^{(\Pc (\ss))} (\alpha_n (\, \cdot \, | x^{(n)}) ; \alpha_n (\, \cdot \, | y^{(n)}))\nonumber\\
 &\quad=  \Wuno^{(\Pc (\ss))} ( \zeta_n^{(x)} (z)  [\alpha+ n \mathfrak{e}_n^{(x)} ] (\, \cdot \,) ;  \zeta_n^{(y)} (z)  [\alpha+ n \mathfrak{e}_n^{(y)} ] (\, \cdot \,) )\nonumber \\
 & \quad= \sup_{g \in \text{Lip}_1} \Big| \int_\ss g (z)  [\zeta_n^{(x)} (z) - \zeta_n^{(y)} (z) ]  \alpha (\ud z )+
 \int_\ss g(z) \zeta_n^{(x)} (z) n \mathfrak{e}_n^{(x)} (\ud z ) -\int_\ss g(z) \zeta_n^{(y)} (z) n \mathfrak{e}_n^{(y)} (\ud z )   \Big| \nonumber \\
 & \quad \leq  \sup_{g \in \text{Lip}_1} \Big| \int_\ss g(z) [\zeta_n^{(x)} (z) - \zeta_n^{(y)} (z) ]  \alpha (\ud z )  \Big| +\sup_{g \in \text{Lip}_1} \Big| \int_\ss g (z)  [\zeta_n^{(x)} (z) - \zeta_n^{(y)} (z) ]   n \mathfrak{e}_n^{(x)} (\ud z ) \Big|\\
 & \quad\quad + \sup_{g \in \text{Lip}_1}  \Big| \int_\ss  g (z)  \zeta_n^{(y)} (z) n [\mathfrak{e}_n^{(x)}-\mathfrak{e}_n^{(y)}] (\ud z)\Big|\nonumber
\end{align}
by an application of the triangular inequality. Observe that the supremum can be equivalently made over all the functions $g \in \text{Lip}_1$ with the property $g(x_0)=0$ for some $x_0 \in \ss$; this is because Wasserstein metric is a distance between probability measures, and by considering the supremum over the sets of functions with a prescribed value at a certain point, such a distance remains the same. We make use of this fact hereafter. 

With regards to the first term $\sup_{g \in \text{Lip}_1} | \int_\ss g(z) [\zeta_n^{(x)} (z) - \zeta_n^{(y)} (z) ]  \alpha (\ud z )  | $ of \eqref{eq:threeterms}, by Fubini-Tonelli theorem
\begin{equation}
\label{eq:A}
\begin{split}
 &  \sup_{g \in \text{Lip}_1} \Big| \int_\ss g(z) [\zeta_n^{(x)} (z) - \zeta_n^{(y)} (z) ]  \alpha (\ud z )  \Big| \\ & \qquad
 =  \frac{1}{n} \sup_{g \in \text{Lip}_1} \Big| \int_0^{+\infty}   G_{g,\alpha, \beta}  (u)  [f_{U_n}^{(x)} (u ) - f_{U_n}^{(y)} (u)]  \ud u\Big|,
 \end{split}
\end{equation}
where $G_{g,\alpha, \beta}  (u) := \int_\ss g (z) u(u +\beta (z))^{-1} \alpha (\ud z )$. For any $g \in \text{Lip}_1$, the following chain of inequalities holds true
\begin{align*}
 & \Big| \int_0^{+\infty}   G_{g,\alpha, \beta}  (u)  [f_{U_n}^{(x)} (u ) - f_{U_n}^{(y)} (u)]  \ud u\Big|\\
 &\quad\leq  \int_0^{+\infty} \Big| \int_\ss g (z)  \frac{u }{u +\beta (z)} \alpha (\ud z ) \Big|  | f_{U_n}^{(x)} (u ) - f_{U_n}^{(y)} (u) | \ud u \\
  & \quad \leq \int_0^{+\infty} \int_\ss |g (z)|  \Big|  \frac{u }{u +\beta (z)} \Big| \alpha (\ud z )  | f_{U_n}^{(x)} (u ) - f_{U_n}^{(y)} (u)|  \ud u   \\
  & \quad \leq \int_0^{+\infty} \int_\ss |g (z)- g(x_0)| \alpha (\ud z )  | f_{U_n}^{(x)} (u ) - f_{U_n}^{(y)} (u) | \ud u   \\
    & \quad \leq \int_0^{+\infty} \int_\ss \ud_{\ss} (z,x_0)  \alpha (\ud z )  | f_{U_n}^{(x)} (u ) - f_{U_n}^{(y)} (u) | \ud u\\
    &\quad  \leq   a  \text{diam} (\ss) \int_0^{+\infty}    | f_{U_n}^{(x)} (u ) - f_{U_n}^{(y)} (u) | \ud u,
\end{align*}
where we used the fact that $g $ is a Lipschitz function satisfying $g(x_0)=0$, and we exploited the assumption that metric space is totally bounded. Thus, we have determined an upper bound for the right-hand side of  \eqref{eq:A}, namely
\begin{equation}
\label{eq:boundA_L1}
   \sup_{g \in \text{Lip}_1} \Big| \int_\ss g(z) [\zeta_n^{(x)} (z) - \zeta_n^{(y)} (z) ]  \alpha (\ud z )  \Big|  \leq 
   \frac{a}{n}\text{diam}(\ss )  ||f_{U_n}^{(x)} - f_{U_n}^{(y)} ||_{1}
\end{equation}
where in Equation \eqref{eq:boundA_L1} we denoted by $||g||_1:= \int_0^{+\infty} |g(u)| \ud u$ the $L_1$-norm of a real-valued function on $\reals^+$. \\

With regards to the second term $\sup_{g \in \text{Lip}_1} | \int_\ss g (z)  [\zeta_n^{(x)} (z) - \zeta_n^{(y)} (z) ]   n \mathfrak{e}_n^{(x)} (\ud z ) |$ of \eqref{eq:threeterms}, by Fubini-Tonelli theorem
\begin{equation}
\label{eq:B}
\begin{split}
&   \sup_{g \in \text{Lip}_1} \Big| \int_\ss g (z)  [\zeta_n^{(x)} (z) - \zeta_n^{(y)} (z) ]   n \mathfrak{e}_n^{(x)} (\ud z ) \Big| \\ & \qquad
   = \sup_{g \in \text{Lip}_1}  \Big| \int_0^{+\infty}  G_{g,n,\beta} (u) [f_{U_n}^{(x)} (u) -f_{U_n}^{(y)}(u)] \ud u \Big|,
   \end{split}
\end{equation}
where $G_{g,n,\beta} (u) := \int_\ss g(x)  u(u +\beta (z))^{-1}  \mathfrak{e}_n^{(x)} (\ud z )$. Along similar lines as before, it can be proved that
\begin{align*}
 & \Big| \int_0^{+\infty}   G_{g,n, \beta}  (u)  [f_{U_n}^{(x)} (u ) - f_{U_n}^{(y)} (u)]  \ud u\Big|\\
 &\quad\leq   \int_0^{+\infty}  |G_{g,n, \beta}  (u)| |f_{U_n}^{(x)} (u ) - f_{U_n}^{(y)} (u)|  \ud u \nonumber \\
 & \quad\leq  \int_0^{+\infty} \int_\ss \ud_\ss (z, x_0)  \mathfrak{e}_n (\ud z )  |f_{U_n}^{(x)} (u ) - f_{U_n}^{(y)} (u)|  \ud u \\
&\quad \leq \text{diam} (\ss ) || f_{U_n}^{(x)} - f_{U_n}^{(y)} ||_1,
\end{align*}
where we exploited the fact that $g $ is a Lipschitz function such that $g(x_0)=0$, and we exploited the assumption that metric space is totally bounded. Thus, we have an upper bound for the right-hand side of  \eqref{eq:B}, that is
\begin{equation}
\label{eq:boundB_L1}
    \sup_{g \in \text{Lip}_1}  \Big| \int_0^{+\infty}  G_{g,n,\beta} (u) [f_{U_n}^{(x)} (u) -f_{U_n}^{(y)}(u)] \ud u \Big|
   \leq\text{diam} (\ss ) || f_{U_n}^{(x)} - f_{U_n}^{(y)} ||_1.
\end{equation}

Finally, with regards to the third term $ \sup_{g \in \text{Lip}_1}  | \int_\ss  g (z)  \zeta_n^{(y)} (z) n [\mathfrak{e}_n^{(x)}-\mathfrak{e}_n^{(y)}] (\ud z)|$ of \eqref{eq:threeterms},  we can write
\begin{align*}
&\sup_{g \in \text{Lip}_1}  \Big| \int_\ss  g (z)  \zeta_n^{(y)} (z) n [\mathfrak{e}_n^{(x)}-\mathfrak{e}_n^{(y)}] (\ud z)\Big|\\
&\quad =\sup_{g \in \text{Lip}_1} \Big| \int_0^{+\infty}   \int_\ss  g(x) \frac{u}{u +\beta (z)} [\mathfrak{e}_n^{(x)}- \mathfrak{e}_n^{(y)}] (\ud z ) f_{U_n}^{(x)} (u) \ud u \Big|
\end{align*}
and bound the absolute value of the right-hand side of the previous expression for any $g \in \text{Lip}_1$ such that $g (x_0) =0$, i.e.
\begin{align} \label{eq:triangularC}
& \Big| \int_0^{+\infty}   \int_\ss  g(z) \frac{u}{u +\beta (z)} [\mathfrak{e}_n^{(x)}- \mathfrak{e}_n^{(y)}] (\ud z ) f_{U_n}^{(x)} (u) \ud u \Big|\\ 
& \quad\leq  \int_0^{+\infty}  \Big| \int_\ss  g(z) \frac{u}{u +\beta (z)} [\mathfrak{e}_n^{(x)}- \mathfrak{e}_n^{(y)}] (\ud z )  \Big| f_{U_n}^{(x)} (u) \ud u \nonumber.
\end{align}
Now, we observe that for any fixed $u \in \reals^+$, the function $z \mapsto G(z):= g(z) u/(u +\beta (z))$ is a Lipschitz function on the domain $\ss$ with Lipschitz constant $1+ \text{diam} (\ss)  L/\beta_0$. In particular, for any $x,y \in \ss$ we have
\begin{align}\label{eq:triangularC11}
|G(x)-G(y)| &= \Big|  g(x)  \frac{u}{u+\beta (x)} - g(y)  \frac{u}{u+\beta (y)}  \Big|\nonumber\\
&  =\Big|  g(x)  \frac{u}{u+\beta (x)} - g(y)  \frac{u}{u+\beta (x)}  + g(y)  \frac{u}{u+\beta (x)}- g(y)  \frac{u}{u+\beta (y)}  \Big|\nonumber\\
& = |g(x)-g(y)| \frac{u}{u+\beta(x)}  +|g(y)| \Big|  \frac{u}{u+\beta (x)}-  \frac{u}{u+\beta (y)}  \Big|\nonumber\\
& \leq  \ud_\ss (x,y) \frac{u}{u +\beta (x)} +|g(y)-g(x_0)|  u  \frac{|\beta (y)-\beta (x)|}{(u+\beta (x))(u+\beta (y))}\nonumber\\
& \leq  \ud_\ss (x,y) \frac{u}{u +\beta (x)} + \ud_\ss (y,x_0)  u  \frac{L \ud_\ss (x,y)}{(u+\beta (x))(u+\beta (y))}\nonumber\\
&  \leq  \ud_\ss (x,y) \frac{u}{u +\beta (x)} + \text{diam} (\ss )  u  \frac{L \ud_\ss (x,y)}{(u+\beta (x))(u+\beta (y))}\nonumber\\
& \leq  \ud_\ss (x,y)  \left[ 1+   \text{diam} (\ss )  \frac{L}{\beta_0}\right].
\end{align}
From \eqref{eq:triangularC} and \eqref{eq:triangularC11},
\begin{align} \label{eq:boundC}
&\sup_{g \in \text{Lip}_1}  \Big| \int_\ss  g (z)  \zeta_n^{(y)} (z) n [\mathfrak{e}_n^{(x)}-\mathfrak{e}_n^{(y)}] (\ud z)\Big|\\
& \quad \leq  \left[ 1+\frac{L \text{diam} (\ss )}{\beta_0}  \right]  \int_0^{+\infty}  \sup_{r \in \text{Lip}_1}  \Big| 
\int_\ss  r(z)  [\mathfrak{e}_n^{(x)}-\mathfrak{e}_n^{(y)}] (\ud z) \Big|   f_{U_n}^{(x)}  (u ) \ud u\nonumber   \\
& \quad =  \left[ 1+\frac{L  \text{diam} (\ss )}{\beta_0}  \right]    \Wc_1^{(\Pc(\ss ))}  (\mathfrak{e}_n^{(x)}; \mathfrak{e}_n^{(y)})\nonumber.
\end{align}

By combining the inequality \eqref{eq:threeterms} with the upper bounds obtained in \eqref{eq:boundA_L1}, \eqref{eq:boundB_L1} and \eqref{eq:boundC} we can write that
\begin{align}\label{eq:Wasserstein_bound_extended}
&\Wc_1^{(\Pc (\ss))} (\alpha_n (\, \cdot \, | x^{(n)}) ; \alpha_n (\, \cdot \, | y^{(n)}))\\
  & \qquad \leq \left[ 1+\frac{L \text{diam} (\ss )}{\beta_0}  \right]   \Wc_1^{(\Pc(\ss ))}  (\mathfrak{e}_n^{(x)};\mathfrak{e}_n^{(y)})
 + \left[\frac{a}{n}+1 \right]\text{diam} (\ss )  || f_{U_n}^{(x)} - f_{U_n}^{(y)} ||_1\nonumber
\end{align}
and, to conclude, it remains to estimate $|| f_{U_n}^{(x)} - f_{U_n}^{(y)} ||_1$. In particular, for any  $x^{(n)}, y^{(n)} \in \ss^n$, by Lemma \ref{lem:L1}
\begin{displaymath}
|| f_{U_n}^{(x)} - f_{U_n}^{(y)} ||_1 \leq  \frac{ 2\beta_1^a aL}{\beta_0^{a+1}}\Wc_1^{(\Pc(\ss))}(\mathfrak{e}_n^{(x)}; \mathfrak{e}_n^{(y)}),
\end{displaymath}
which leads to 
\begin{align}\label{final_negp}
&\Wc_1^{(\Pc (\ss))} (\alpha_n (\, \cdot \, | x^{(n)}) ; \alpha_n (\, \cdot \, | y^{(n)}))\\
  & \qquad \leq   \left\{\left[ 1+\frac{L \cdot \text{diam} (\ss )}{\beta_0}  \right] 
 + \left[\frac{a}{n}+1 \right]\text{diam} (\ss ) \frac{ 2\beta_1^a aL}{\beta_0^{a+1}}  \right\} \Wc_1^{(\Pc(\ss))}(\mathfrak{e}_n^{(x)}; \mathfrak{e}_n^{(y)})\nonumber.   
\end{align}
The inequality \eqref{final_negp} entails that Proposition \ref{main_prop2} is applicable for a suitable choice of $\bar{L}$. Then, Theorem \ref{main_thm2} holds.


\section{Discussion}\label{discuss}

We introduced a general approach to provide PCRs in Bayesian nonparametric models where the posterior distribution is available through a more general disintegration than the Bayes formula, and hence models that are non-dominated for the observations. Some refinements of our main result are presented under additional prior assumptions, showing the critical role of the predictive distributions for establishing PCRs. To the best of our knowledge, this is the first general approach to provide PCRs in non-dominated Bayesian nonparametric models, and it paves the way to further work along this line of research. An interesting problem is to extend Proposition \ref{main_prop2} to deal with nonparametric priors whose predictive distributions do not depend uniquely on the empirical process of the observations. The Pitman-Yor process prior \cite{Per(92),Pit(97)} is arguably the most popular example of such a class of priors, with predictive distributions depending on both the empirical process of the observations and the number of distinct types among the observations \citep{Pit(03),Deb(15),Bac(17)}. Another interesting problem is to extend Theorem \ref{main_thm2} to deal with right-censored survival times, and then consider the use of nonparametric priors such as the Beta-Stacy process prior \cite{Wal(97)} and generalizations thereof in the class of neutral to the right process priors \citep{Lij(10)}. Finally, it is worth mentioning that our approach can be also applied to Bayesian nonparametric models that are dominated for the observations. In particular, any possible use of the Bayes formula should be seen as a mathematical tool to obtain the continuity assumption \eqref{post_continuity} through the arguments exposed in \cite[ Sections 2.3, Section 2.4 and Section 4.1]{DM(20b)}. Our approach may prove to be effective in such a context, especially in the case of real-valued observations. Work on this is ongoing.


\appendix

\section{Auxiliary lemmas of Section \ref{sect:main_problem}}

\begin{lemma}\label{lem_pcr}
Assume that $\pi \in \mathcal P_p(\ps)$ and that, for any $n \in \naturals$, $\pfrak_0^{\otimes_n} \ll \mu_n$. Then, 
$\pi_n(\cdot| \xi_1, \dots, \xi_n)$ is a well-defined random probability measure belonging to $\mathcal P_p(\ps)$ with $\pp$-probability one, and \eqref{Wpcr}
gives a PCR at $\pfrak_0$. 
\end{lemma}

\begin{proof}
Recalling that $(\ps,\psa)$ is a Borel space, we have that any two solutions $\pi_n(\cdot|\cdot)$ and $\pi'_n(\cdot|\cdot)$ of \eqref{disintegration} satisfy $\pi_n(\cdot|x^{(n)}) = \pi'_n(\cdot|x^{(n)})$, as elements of
$\pms$, for all $x^{(n)} \in \ss^n\setminus N_n$, where $N_n$ is a $\mu_n$-null set. Thus, the assumption $\pfrak_0^{\otimes_n} \ll \mu_n$ entails that $\xi^{(n)} := (\xi_1, \dots, \xi_n)$ takes values in $N_n$ with $\pp$-probability zero, yielding the desired well-definiteness. If $\pi \in \mathcal P_p(\ps)$, any solution $\pi_n(\cdot|\cdot)$ of \eqref{disintegration} satisfies $\pi_n(\mathcal P_p(\ps)|x^{(n)}) = 1$ for almost every $x^{(n)}\in\ss^n$. Whence, $\pi_n(\mathcal P_p(\ps)|\xi^{(n)}) = 1$ $\pp$\text{-a.s.}, yielding
$$
\Wp^{(\ppms)}(\pi_n(\cdot|\xi^{(n)}); \delta_{\pfrak_0}) = \left( \int_{\ps} \left[ \Wp^{(\pms)}(\pfrak,\pfrak_0) \right]^p \pi_n(\ud\pfrak|\xi^{(n)}) \right)^{1/p}
$$
is a random variable which is $\pp$\text{-a.s.} finite. Then, combining Markov's and Lyapunov's inequalities we have that
$$
\pi_n\left(\left\{\pfrak \in \ps\ :\ \Wp^{(\pms)}(\pfrak,\pfrak_0)  \geq M_n\epsilon_n\right\} \big| \xi^{(n)} \right) \leq 
\frac{ \left(\int_{\ps} \left[ \Wp^{(\pms)}(\pfrak,\pfrak_0) \right]^p \pi_n(\ud\pfrak|\xi^{(n)})\right)^{1/p} }{M_n \epsilon_n}
$$
holds $\pp$\text{-a.s.}. After taking expectation of both sides of the previous inequality, and recalling \eqref{Wpcr}, we conclude that 
$$
\ee\left[ \pi_n\left(\left\{\pfrak \in \ps\ :\ \Wp^{(\pms)}(\pfrak,\pfrak_0) \geq M_n\epsilon_n\right\} \big| \xi^{(n)} \right) \right] \leq \frac{1}{M_n} \rightarrow 0
$$
holds true for any diverging sequence $\{M_n\}_{n \geq 1}$, which is tantamount to saying that $\epsilon_n$ is a PCR at $\pfrak_0$. 
\end{proof}

\section{Proofs of Section \ref{sec_mainres}}

\subsection{Proofs of Corollary \ref{main_cor2}}\label{proof_cor}

To prove the corollary, put $\delta \sim n^{-\beta}$ for some $\beta > 0$, so that $N_{\delta_n} \sim n^{d\beta}$. Therefore, leaving the Glivenko-Cantelli-term $\varepsilon_{n,p}$ out of these computations,
check that the term $2(2 + L_n)\delta_n$ is asymptotic to $n^{s-\beta}$, as $n\to +\infty$. Taking account of \eqref{asymptoticMV}, the last term on the right-hand side of \eqref{main_bound2} has an upper bound 
which is asymptotic to $n^{-(\alpha - d\beta)/p}$. At this stage, by equalizing the two exponents, that is by imposing
$$
\beta - s = \frac{\alpha - d\beta}{p}\ ,
$$
yields to the identity $\beta = (\alpha + ps)/(d+p)$. Moreover, note that assumption $\alpha> sd$ entails that $(\alpha + ps)/(d+p)\in (s, \alpha/d)$. Accordingly, the sum of the last two terms on the right-hand side of \eqref{main_bound2} is bounded by a term which is (globally) asymptotically equivalent to $n^{- (\alpha - sd)/(d+p)}$. This completes the proof. 

\subsection{An estimate of $M_n(N_{\delta}, \delta) + V_n(N_{\delta}, \delta)$}\label{ineq}

Following ideas originally developed in \cite{DF(90)}, let $\chi$ be a probability measure on $[0,1]$. For any $h \in (0, \frac 14)$, put 
\begin{equation*} \label{phiDF}
\phi(h) := \inf_{p \in [0,1]} \chi([0,1] \cap [p, p+h])
\end{equation*}
and, for any $p \in [0,1]$,
$$
R(n,p,h) := \frac{\int_{[0,1] \cap [p-h, p+h]} t^{np} (1-t)^{n(1-p)} \chi(\ud t)}{\int_{[0,1] \cap [p-h, p+h]^c} t^{np} (1-t)^{n(1-p)} \chi(\ud t)}\ . 
$$
Note that, if $\inf_{h \in (0, \frac 14)} \phi(h) > 0$, then the above ratio is well-defined since the denominator is positive. In addition, put
$$
H(p,t) := -p\log t - (1-p) \log (1-t)
$$
and 
$$
g(h) := \inf_{\substack{p,t \in [0,1], \\ |p-t| \geq h}} \{H(p,t) - H(p,p)\}\ . 
$$
Finally, setting
$$
h^{\ast} := \min\{h, 2^{-1} g(h) [g(h) - 2h^2]\}
$$
and 
\begin{equation} \label{psiDF}
\psi(h) := \phi(h^{\ast})\ ,
\end{equation}
there holds
\begin{equation} \label{DFinequality}
R(n,p,h) \geq \psi(h) e^{2nh^2}
\end{equation}
for all $n \in \naturals$, all $p \in [0,1]$ and all $h \in (0, \frac 14)$. See \cite{DF(90)} for details. Hereafter, we show how to use inequality \eqref{DFinequality} to get an upper bound for $M_n(N_{\delta}, \delta) + V_n(N_{\delta}, \delta)$ as in \eqref{asymptoticMV}. It is sufficient to deal with $M_n(N,\delta)$, as the treatment of  $V_n(N,\delta)$ is  analogous. In particular, for fixed $j \in \{1, \dots, N\}$, we have that
$$
|\phi_{j,n}^{(\xi)} - m_{j,n}^{(\xi)}(\pi_{N,j})| \leq h + \frac{1}{R(n, \phi_{j,n}^{(\xi)}, h)} \leq h + \frac{e^{-2nh^2}}{\psi(h)}
$$
for any $h \in (0, \frac 14)$, and then we can minimize the upper bound of the last expression with respect to $h$, thus optimizing the upper bound. In particular, if $\psi(h) \sim h^s$ as $h \to 0^+$, for some $s>0$ possibly depending on $N$, then the upper bound is minimized by a term that is asymptotically equivalent to $\sqrt{(s+1)\log n/n}$. 

\subsection{Proof of Proposition \ref{main_prop2}}\label{proof_prp2}

First, if there exists a distinguished solution $\pi_n^{\ast}(\cdot|\cdot)$ of \eqref{disintegration} satisfying \eqref{post_continuity} for all $x^{(n)}, y^{(n)} \in \ss^n$, then we can define
$$
\alpha_n^{\ast}(A|x^{(n)}) := \int_{\pms} \pfrak(A)\pi_n^{\ast}(\ud \pfrak|x^{(n)})
$$
and note that $\alpha_n^{\ast}(\cdot|\cdot)$ fulfills both \eqref{predictive}, by Fubini's theorem, and \eqref{Lipschitz_predictive}, by the convexity of $\Wuno$ \citep[Chapter 7]{Vill(03)}. On the other hand, we now suppose that \eqref{Lipschitz_predictive} is in force for some (unique) probability kernel $\alpha_n^{\ast}(\cdot|\cdot)$ satisfying \eqref{predictive}. The key observation is that the posterior distribution
can be obtained as weak limit of
$$
\pp\left[\frac 1m \sum_{i=1}^m \delta_{X_{i+n}} \in \cdot\Big |X_1, \dots, X_n\right]
$$
as $m \rightarrow +\infty$, by an application of de Finetti's theorem. Therefore, in view of the well-known Kantorovic-Rubinstein dual representation \citep[Chapter 11]{du}, the thesis to be proved is implied by the validity of
\begin{align} 
\sup_{\substack{\varphi : \mathcal P_1(\ss) \to \reals \\ 1\text{-Lipschitz}(\Wuno)}} & \Big|\ee\left[\varphi\left(\frac 1m \sum_{i=1}^m \delta_{X_{i+n}}\right)\ \Big| X^{(n)}=x^{(n)} \right] \nonumber \\
&\quad- \ee\left[\varphi\left(\frac 1m \sum_{i=1}^m \delta_{X_{i+n}}\right)\ \Big| X^{(n)}=y^{(n)} \right] \Big| \leq \bar{L} \ \Wuno^{(\pms)}(\empiric^{(x)}; \empiric^{(y)}) \label{dudley}
\end{align}
for any $n,m \in \naturals$ and $x^{(n)}, y^{(n)} \in \ss^n$, where the notation $1\text{-Lipschitz}(\Wuno)$ indicates that $|\varphi(\pfrak_1) - \varphi(\pfrak_2)| \leq \Wuno^{(\pms)}(\pfrak_1; \pfrak_2)$ for all
$\pfrak_1, \pfrak_2 \in \mathcal P_1(\ss)$. An equivalent formulation could be given in terms of the $m$-predictive distributions on the quotient metric space $(\ss^m\setminus\sim, \tilde{\ud}_{\ss}^m)$, say $\alpha_n^{(m)}(\cdot|X_1, \dots, X_n)$, which is given by
$$
\alpha_n^{(m)}(C_m|X_1, \dots, X_n) := \pp\left[\frac 1m \sum_{i=1}^m \delta_{X_{i+n}} \!\!\in \mathfrak{e}_m\left(\text{quotient}^{-1}(C_m)\right)\Big |X^{(n)}\right]
$$
for any $C_m \in \Bcr(\ss^m\setminus\sim)$. Hereafter, we prove \eqref{dudley} by induction on $m\geq1$. In particular, for $m=1$ the inequality \eqref{dudley} boils down to \eqref{Lipschitz_predictive}, where we have chosen just $\alpha_n^{\ast}(\cdot|\cdot)$ as representative, according to
$$
\ee\left[\varphi(\delta_{X_{n+1}})\ \Big| X^{(n)}=x^{(n)} \right] = \int_{\ss} \varphi(\delta_z) \alpha_n^{\ast}(\ud z|x^{(n)})\ . 
$$
Then, we assume that \eqref{dudley} is true for some $m$ and we prove that it is true for $m+1$. We denote by $\alpha_{n,m}^{\ast}(\cdot|\cdot)$ a version of the $m$-predictive distribution (i.e., the conditional distributions of $(X_{n+1}, \dots, X_{n+m})$ given $(X_1, \dots, X_n)$) for which \eqref{dudley} is true for any $x^{(n)}, y^{(n)} \in \ss^n$.
Hence, we notice that we can exploit the tower property of the predictive distributions to select $\alpha_{n,m+1}^{\ast}(\cdot|\cdot)$ so that
\begin{align*}
&\ee\left[\varphi\left(\frac{1}{m+1} \sum_{i=1}^{m+1} \delta_{X_{i+n}}\right)\ \Big| X^{(n)}=x^{(n)} \right] \\
&\quad= \int_{\ss^{m+1}} \varphi\left(\frac{1}{m+1} \sum_{i=1}^{m+1} \delta_{z_i}\right) \alpha_{n,m+1}^{\ast}(\ud z_1, \dots, \ud z_{m+1}\ |x^{(n)})\\
&\quad= \int_{\ss^m} \int_{\ss} \varphi\left(\frac{1}{m+1} \sum_{i=1}^{m+1} \delta_{z_i}\right) \alpha_{n+1,m}^{\ast}(\ud z_1, \dots, \ud z_m\ |x^{(n)}, z_{m+1}) \alpha_{n,1}^{\ast}(\ud z_{m+1}\ |x^{(n)}) \ .
\end{align*}
At this stage, the left-hand side of \eqref{dudley} with $m+1$ in place of $m$ can be bounded by the supremum, with the supremum taken over those functions $\varphi : \mathcal P_1(\ss) \to \reals$  
which are 1-Lipschitz$(\Wuno)$, of the following
\begin{align*} 
&\Big| \int_{\ss^m} \int_{\ss} \varphi\left(\frac{1}{m+1} \sum_{i=1}^{m+1} \delta_{z_i}\right) \alpha_{n+1,m}^{\ast}(\ud z_1, \dots, \ud z_m\ |x^{(n)}, z_{m+1}) \alpha_{n,1}^{\ast}(\ud z_{m+1}\ |x^{(n)}) \\
& \quad-  \int_{\ss^m} \int_{\ss} \varphi\left(\frac{1}{m+1} \sum_{i=1}^{m+1} \delta_{z_i}\right) \alpha_{n+1,m}^{\ast}(\ud z_1, \dots, \ud z_m\ |x^{(n)}, z_{m+1}) \alpha_{n,1}^{\ast}(\ud z_{m+1}\ |y^{(n)}) \Big| \\
&\quad\quad+ \Big|\int_{\ss^m} \int_{\ss} \varphi\left(\frac{1}{m+1} \sum_{i=1}^{m+1} \delta_{z_i}\right) \alpha_{n+1,m}^{\ast}(\ud z_1, \dots, \ud z_m\ |x^{(n)}, z_{m+1}) \alpha_{n,1}^{\ast}(\ud z_{m+1}\ |y^{(n)}) \\
&\quad\quad\quad -  \int_{\ss^m} \int_{\ss} \varphi\left(\frac{1}{m+1} \sum_{i=1}^{m+1} \delta_{z_i}\right) \alpha_{n+1,m}^{\ast}(\ud z_1, \dots, \ud z_m\ |y^{(n)}, z_{m+1}) \alpha_{n,1}^{\ast}(\ud z_{m+1}\ |y^{(n)}) \Big| 
\end{align*}
which, for notational ease, we shorten as $|A| + |B|$. With regards to the term $|B|$, we observe that it holds
\begin{align*} 
|B| &\leq \int_{\ss} \Big| \int_{\ss^m} \varphi\left(\frac{1}{m+1} \sum_{i=1}^{m+1} \delta_{z_i}\right) \alpha_{n+1,m}^{\ast}(\ud z_1, \dots, \ud z_m\ |x^{(n)}, z_{m+1}) \\
&\quad- \int_{\ss^m} \varphi\left(\frac{1}{m+1} \sum_{i=1}^{m+1} \delta_{z_i}\right) \alpha_{n+1,m}^{\ast}(\ud z_1, \dots, \ud z_m\ |y^{(n)}, z_{m+1}) \Big|\ \alpha_{n,1}^{\ast}(\ud z_{m+1}\ |y^{(n)}) 
\end{align*}
and, since the mapping $\ss^m\setminus\!\!\sim\ \ni [z_1, \dots, z_m] \mapsto \sum_{i=1}^{m+1} \delta_{z_i}$ is $\frac{m}{m+1}$-Lipschitz for any fixed $z_{m+1}$, we write
\begin{align}
|B| &\leq L \frac{m}{m+1} \Wuno\left(\frac{1}{n+1}\left[\delta_{z_{m+1}} + \sum_{i=1}^n \delta_{x_i}\right]; \frac{1}{n+1}\left[\delta_{z_{m+1}} + \sum_{i=1}^n \delta_{y_i}\right]\right) \nonumber \\
& = \bar{L}\frac{m}{m+1} \frac{n}{n+1}\Wuno(\empiric^{(x)}; \empiric^{(y)}) \label{dudley_B}
\end{align}
by the convexity of $\Wuno$ \citep[Chapter 7]{Vill(03)}. With regards to the term $A$, we notice that it can be written as follows
$$
\int_{\ss} \Phi_{x^{(n)},n}(z) \alpha_{n,1}^{\ast}(\ud z\ |y^{(n)}) - \int_{\ss} \Phi_{x^{(n)},n}(z) \alpha_{n,1}^{\ast}(\ud z\ |x^{(n)})
$$ 
with
$$
\Phi_{x^{(n)},n}(z) := \int_{\ss^m} \varphi\left(\frac{1}{m+1} \left[\delta_z + \sum_{i=1}^{m} \delta_{z_i}\right]\right) \alpha_{n+1,m}^{\ast}(\ud z_1, \dots, \ud z_m\ |x^{(n)}, z)\ .
$$
If we write
\begin{align*} 
&\big| \Phi_{x^{(n)},n}(u) - \Phi_{x^{(n)},n}(v) \big| \\
&\quad\leq \Big| \int_{\ss^m} \varphi\left(\frac{1}{m+1} \left[\delta_u + \sum_{i=1}^{m} \delta_{z_i}\right]\right) \alpha_{n+1,m}^{\ast}(\ud z_1, \dots, \ud z_m\ |x^{(n)}, u) \\
&\quad\quad- \int_{\ss^m} \varphi\left(\frac{1}{m+1} \left[\delta_u + \sum_{i=1}^{m} \delta_{z_i}\right]\right) \alpha_{n+1,m}^{\ast}(\ud z_1, \dots, \ud z_m\ |x^{(n)}, v) \Big| \\
&\quad\quad\quad+ \int_{\ss^m} \Big|\varphi\left(\frac{1}{m+1} \left[\delta_u + \sum_{i=1}^{m} \delta_{z_i}\right]\right) - \varphi\left(\frac{1}{m+1} \left[\delta_v + \sum_{i=1}^{m} \delta_{z_i}\right]\right) \Big|\alpha_{n+1,m}^{\ast}(\ud z_1, \dots, \ud z_m\ |x^{(n)}, v)\ ,
\end{align*}
we conclude that, uniformly in $x^{(n)} \in \ss^n$, the function $z \mapsto \Phi_{x^{(n)},n}(z)$ is $\left(\frac{m}{m+1} \frac{1}{n+1} + \frac{1}{m+1}\right)$-Lipschitz, so that
$$
|A| \leq \bar{L} \left( \frac{m}{m+1} \frac{1}{n+1} + \frac{1}{m+1}\right)\Wuno(\empiric^{(x)}; \empiric^{(y)}) \ .
$$
Combining this last inequality with \eqref{dudley_B} yields \eqref{dudley} for any $m \in \naturals$ and, hence, \eqref{post_continuity}. This completes the proof.

\section{Auxiliary lemmas of Section \ref{sect:illustrations}}

\begin{lemma} \label{lem:L1}
For any $x^{(n)}, y^{(n)} \in \ss^n$ the $L_1$ distance between the density functions $f_{U_n}^{(x)}$ and $f_{U_n}^{(y)}$ is such that
\begin{equation}
\label{eq:L1_bound}
|| f_{U_n}^{(x)} - f_{U_n}^{(y)} ||_1 \leq  \frac{ 2\beta_1^a aL}{\beta_0^{a+1}}\Wc_1^{(\Pc(\ss))}(\mathfrak{e}_n^{(x)}; \mathfrak{e}_n^{(y)}).
\end{equation}
\end{lemma}

\begin{proof}
If we define $I_{n-1}^{(x)}  (u)  := u^{n-1} \exp \{ -\int_\ss  \log (u+\beta (z))[\alpha + n \mathfrak{e}_n^{(x)}](\ud z)\}$, then we have to upper bound
\begin{align} \label{eq:I_bound}
|| f_{U_n}^{(x)} - f_{U_n}^{(y)} ||_1 & = \int_0^\infty  \Big| \frac{I_{n-1}^{(x)}  (u)}{\int_0^\infty  I_{n-1}^{(x)}  (u) \ud u } -\frac{I_{n-1}^{(y)}  (u)}{\int_0^\infty  I_{n-1}^{(y)}  (u) \ud u }\Big| \ud u\nonumber\\
& \leq \frac{2}{\int_0^\infty  I_{n-1}^{(x)}(u) \ud u } \int_0^\infty \Big| I_{n-1}^{(x)}  (u)- I_{n-1}^{(y)}  (u) \Big|   \ud u,
\end{align}
where the last inequality follows by adding and subtracting $I_{n-1}^{(y)}  (u)/\int_0^\infty I_{n-1}^{(x)}  (u) \ud u$. Now, we upper bound and lower bound the numerator and the denominator of \eqref{eq:I_bound}, respectively. With regards to the numerator of \eqref{eq:I_bound}, 
\begin{align}\label{eq:exponential_diff}
&\int_0^\infty \Big| I_{n-1}^{(x)}  (u)- I_{n-1}^{(y)}  (u) \Big|   \ud u\\
&\quad   = \int_0^\infty  u^{n-1}\exp \left\{  -\int_\ss \log (u +\beta (z))  \alpha (\ud z ) \right\}\nonumber\\
& \quad\quad \times\Big|\exp \left\{ -\int_\ss  \log (u+\beta (z)) n \mathfrak{e}_n^{(x)} (\ud z ) \right\} - \exp \left\{- \int_\ss  \log (u+\beta (z)) n \mathfrak{e}_n^{(y)} (\ud z )  \right\}  \Big| \ud u \nonumber\\
& \quad \leq \int_0^\infty  \frac{u^{n-1}}{(u +\beta_0)^a} \nonumber\\
& \quad\quad\times \Big|\exp \left\{ -\int_\ss  \log (u+\beta (z)) n \mathfrak{e}_n^{(x)} (\ud z ) \right\} - \exp \left\{- \int_\ss  \log (u+\beta (z)) n \mathfrak{e}_n^{(y)} (\ud z )  \right\}  \Big| \ud u\nonumber . 
\end{align}
In order to bound the difference between the two exponential functions, fix $u \in \reals^+$ and assume, without loss of generality, that $\int_\ss  \log (u+\beta (z)) n \mathfrak{e}_n^{(x)} (\ud z )<\int_\ss  \log (u+\beta (z)) n \mathfrak{e}_n^{(y)} (\ud z )$. Thus, we can write
\begin{align*}
&\Big|\exp \left\{ -\int_\ss  \log (u+\beta (z)) n \mathfrak{e}_n^{(x)} (\ud z ) \right\} - \exp \left\{- \int_\ss  \log (u+\beta (z)) n \mathfrak{e}_n^{(y)} (\ud z )  \right\}  \Big| \\
& \quad =  \exp \left\{ - \int_\ss  \log (u+\beta (z)) n \mathfrak{e}_n^{(x)} (\ud z ) \right\}\Big| 1-
\exp \left\{ - \left[\int_\ss  \log (u+\beta (z)) n [\mathfrak{e}_n^{(y)}- \mathfrak{e}_n^{(x)}] (\ud z ) \right] \right\} \Big|\\
&\quad\leq \exp \left\{ - \int_\ss  \log (u+\beta (z)) n \mathfrak{e}_n^{(x)} (\ud z ) \right\}  \Big| \int_\ss \log (u+\beta (z))  n [\mathfrak{e}_n^{(y)}- \mathfrak{e}_n^{(x)}] (\ud z )\Big|.
\end{align*}
Now, recall that the function $\ss \ni x \longmapsto \log (u+\beta (x)) \in \reals^+$ is Lipschitz with Lipschitz constant 
$L/(\beta_0+u)$, where $u \in \reals^+$ is a fixed quantity. Therefore, we can continue with the previous upper bound to get
\begin{align}\label{eq:exponential_diff11}
&\Big|\exp \left\{ -\int_\ss  \log (u+\beta (z)) n \mathfrak{e}_n^{(x)} (\ud z ) \right\} - \exp \left\{- \int_\ss  \log (u+\beta (z)) n \mathfrak{e}_n^{(y)} (\ud z )  \right\}  \Big|\\
&\quad \leq \exp \left\{ - \int_\ss  \log (u+\beta (z)) n \mathfrak{e}_n^{(x)} (\ud z ) \right\} n \frac{L}{u +\beta_0} \Wc_1^{(\Pc(\ss))}(\mathfrak{e}_n^{(x)}; \mathfrak{e}_n^{(y)})\nonumber \\
& \quad = \frac{1}{\prod_{i=1}^n (u+\beta (X_i))} \frac{nL}{u +\beta_0} \Wc_1^{(\Pc(\ss))}(\mathfrak{e}_n^{(x)}; \mathfrak{e}_n^{(y)})\nonumber\\
&\quad \leq \frac{nL}{(u +\beta_0)^{n+1}} \Wc_1^{(\Pc(\ss))}(\mathfrak{e}_n^{(x)}; \mathfrak{e}_n^{(y)})\nonumber.
\end{align} 
From \eqref{eq:exponential_diff} and \eqref{eq:exponential_diff11},
\begin{align}\label{eq:num}
&\int_0^\infty \Big| I_{n-1}^{(x)}  (u)- I_{n-1}^{(y)}  (u) \Big|   \ud u\\
&\quad \leq  \Wc_1^{(\Pc(\ss))}(\mathfrak{e}_n^{(x)}; \mathfrak{e}_n^{(y)}) \int_0^\infty  \frac{u^{n-1}}{(u +\beta_0)^a} \frac{nL}{(u+\beta_0)^{n+1}}  \ud u \nonumber \\
& \quad = nL\Wc_1^{(\Pc(\ss))}(\mathfrak{e}_n^{(x)}; \mathfrak{e}_n^{(y)}) \int_0^\infty  \frac{u^{n-1}}{(u +\beta_0)^{a+1+n}}   \ud u \nonumber \\
& \quad= \frac{nL}{\beta_0^{a+1}}\Wc_1^{(\Pc(\ss))}(\mathfrak{e}_n^{(x)}; \mathfrak{e}_n^{(y)})  B (a+1, n) =
\frac{nL}{\beta_0^{a+1}}\Wc_1^{(\Pc(\ss))}(\mathfrak{e}_n^{(x)}; \mathfrak{e}_n^{(y)})  	\frac{a}{a+n}  B (a,n) \nonumber\\
& \quad \leq 
\frac{aL}{\beta_0^{a+1}}\Wc_1^{(\Pc(\ss))}(\mathfrak{e}_n^{(x)}; \mathfrak{e}_n^{(y)})   B (a,n).\nonumber  
\end{align}
This completes the study of the numerator of \eqref{eq:I_bound}. With regards to the denominator of \eqref{eq:I_bound}, we can write
\begin{align}\label{eq:lowdenominator}
\int_0^\infty  I_{n-1}^{(x)}(u) \ud u &=  \int_0^\infty  u^{n-1} \exp \left\{ -\int_\ss \log (u+\beta (z))  [\alpha+ n \mathfrak{e}_n^{(x)}] (\ud z ) \right\}  \ud u \\
&\geq \int_0^\infty  u^{n-1}  e^{-(a+n)\log (u+\beta_1)}  \ud u\nonumber\\
& = \int_0^\infty  \frac{u^{n-1}}{(u+\beta_1)^{a+n}} \ud u\nonumber \\
&= \frac{B(a,n)}{\beta_1^a}.\nonumber
\end{align}
The proof is completed by combining the inequality \eqref{eq:I_bound} with the upper bound \eqref{eq:num} and the lower bound \eqref{eq:lowdenominator}.
\end{proof}

\section*{Acknowledgements}

Federico Camerlenghi, Emanuele Dolera and Stefano Favaro received funding from the European Research Council (ERC) under the European Union's Horizon 2020 research and innovation programme under grant agreement No 817257. The authors acknowledge the financial support from the Italian Ministry of Education, University and Research (MIUR), ``Dipartimenti di Eccellenza" grant 2018-2022. Federico Camerlenghi and Emanuele Dolera are members of the \textit{Gruppo Nazionale per l’Analisi Matematica, la Probabilità e le loro Applicazioni} (GNAMPA) of the \textit{Istituto Nazionale di Alta Matematica} (INdAM).


\end{document}